\documentclass{article}

\usepackage[a4paper,top=2cm,bottom=2cm,left=3cm,right=3cm,marginparwidth=1.75cm]{geometry}

\usepackage{amsmath}
\usepackage{amsthm}
\usepackage{amssymb}
\usepackage{graphicx}
\usepackage[colorlinks=true, allcolors=blue]{hyperref}
\usepackage{wrapfig}
\usepackage{authblk}
\usepackage{subfig}
\usepackage[width=0.8\textwidth]{caption}
\usepackage{tikz}
\usetikzlibrary{arrows.meta,decorations.pathmorphing}
\usepackage{algorithm}
\usepackage{algpseudocode}
\usepackage[section]{placeins}
\tikzset{
    vtx/.style={circle, fill=black, inner sep=1.2pt},
    lbl/.style={font=\footnotesize},
    layer/.style={fill=gray!12, draw=gray!45, rounded corners=5pt},
    lowregion/.style={fill=gray!4, draw=gray!45, rounded corners=5pt},
    newedge/.style={blue!70!black, line width=1.1pt, -{Stealth[length=2.3mm]}},
    newedgeU/.style={blue!70!black, line width=1.1pt},
    crossedge/.style={red!75!black, dashed, line width=0.9pt, -{Stealth[length=2.3mm]}},
    oldpath/.style={black, densely dotted, line width=1.1pt},
    treepath/.style={gray!70!black, line width=0.9pt, decorate,
                     decoration={snake, amplitude=0.55mm, segment length=2.6mm}},
    restC/.style={gray!60, densely dotted, line width=0.9pt},
    fiber/.style={draw=black!45, densely dashed, rounded corners=4pt},
}

\title{On the Girth of Graph Lifts}
\author{Shlomo Hoory}
\affil{{\small Department of Computer Science, Tel-Hai University, Kiryat Shmona, Israel\\hooryshl@telhai.ac.il}}

\newtheorem{theorem}{Theorem}[section]
\newtheorem{lemma}[theorem]{Lemma}
\newtheorem{proposition}[theorem]{Proposition}
\newtheorem{corollary}[theorem]{Corollary}
\newtheorem{remark}[theorem]{Remark}

\newtheorem{claim}{Claim}
\newenvironment{claimproof}[1][Proof of Claim]{\begin{proof}[#1]}{\end{proof}}

\newcommand{\R}{\mathbb{R}}  
\newcommand{\tH}{\tilde{H}}
\newcommand{\dist}{\mbox{dist}}
\newcommand{\diam}{\mbox{diam}}
\newcommand{\girth}{\mbox{girth}}
\newcommand{\dirE}{\overline{E}}

\newcommand{\perm}{\mbox{perm}}

\newcommand{\vgeq}{\mbox{\rotatebox[origin=c]{-90}{$\geq$}}}
\newcommand{\MnR}{M_n(\R)} 
\newcommand{\one}{\mathbf{1}} 
\newcommand{\x}{\mathbf{x}} 
\newcommand{\y}{\mathbf{y}} 
\newcommand{\z}{\mathbf{z}} 
\newcommand{\deltav}{\mathbf{\delta_v}} 

\DeclareMathOperator{\mindeg}{mindeg}
\DeclareMathOperator{\maxdeg}{maxdeg}
\DeclareMathOperator{\tp}{top}

\begin{document}
\maketitle

\begin{abstract}
    The order of the smallest $k$-regular graph of girth at least $g$ is denoted by the well-studied function $n(k,g)$.
    We introduce an analogous function $n(H,g)$, defined as the smallest order of a graph of girth at least $g$ that is a lift (or cover) of the, possibly non-regular, graph $H$.
    We prove that the two main combinatorial bounds on $n(k,g)$--- the Moore lower bound and the Erd\H{o}s-Sachs upper bound---carry over to the new lift setting. 
    
    We also consider two other functions:
    i) The smallest order of a graph of girth at least $g$ sharing a universal cover with $H$. We prove that it is the same as $n(H,g)$ up to a multiplicative constant.
    ii) The smallest order of a graph of girth at least $g$ with a prescribed degree distribution. We discuss this known generalization and argue that the new suggested definitions are superior.
    
    We conclude with the best-known upper and lower bounds for covers of the diamond graph ($K_4 - e$) for girth up to $30$, obtained via a computerized search, and offer conjectures and open problems for future research.
\end{abstract}

\noindent {\bf MSC2020:} Primary 05C35, Secondary 05C50, 05C70, 90C27

\section{Introduction}

The girth of a graph is the length of its shortest cycle. It is a classical challenge to determine $n(k, g)$, the smallest possible number of vertices $n$ of a $k$-regular graph of girth at least $g$. Minimal order $k$-regular graphs of girth $g$ are called $(k,g)$-cages, see~\cite{biggs1993algebraic,exoo2012dynamic} for a survey of the subject.
There are two simple combinatorial arguments yielding lower and upper bounds on $n(k,g)$, namely the Moore lower bound and the Erd\H{o}s-Sachs 1963 upper bound,~\cite{erdos1963regulare}. Their asymptotic version is:
\begin{equation}\label{eqn:regularAsymp}
\Omega({(k-1)}^{g/2}) \leq n(k,g) \leq O({(k-1)}^g).    
\end{equation}

To this day, the Moore bound remains the best asymptotic lower bound, and Erd\H{o}s-Sachs achieves the best asymptotic upper bound by means of a combinatorial construction. There are number theoretic constructions that do better, such as the celebrated 1988 construction by Lubotzky, Phillips and Sarnak~\cite{lubotzky1988ramanujan}, that improves the upper bound to $O({(k-1)}^{3g/4})$ for $k=p+1$ with $p \equiv 1 \pmod{4}$ prime.

Less is known about the minimal order of non-regular graphs of girth at least $g$. 
The result of Alon, Hoory and Linial 2002~\cite{alon2002moore} generalizes the Moore bound to non-regular graphs. 
They proved that the smallest possible number of vertices of a graph of girth at least $g$ with minimal degree at least two and degree distribution ${\cal D}$ is lower bounded by:
\begin{equation}\label{eqn:AsympAHL}
n({\cal D},g) \geq \Omega(\Lambda({\cal D})^{g/2}),
\end{equation}
where $\Lambda = \Lambda({\cal D}) = \prod_{d} {(d-1)}^{(d \, p_d)/{(\sum_{d'} d' p_{d'}})}$ and $p_d$ is the portion of degree $d$ vertices specified by ${\cal D}$.
The base of the exponent, $\Lambda$, is some weighted average of the degree minus one, known to be lower bounded by the average degree minus one.

However, accepting $n({\cal D},g)$ as the right analogue of $n(k,g)$ to non-regular graphs has several drawbacks. 
Firstly, the Erd\H{o}s-Sachs construction does not carry over to the new setting, so there is no "trivial" upper bound. 
Secondly, the existence of Moore graphs, achieving tight equality with the non-asymptotic version of \eqref{eqn:AsympAHL}, does not seem possible, except for degenerate cases, as indicated by Eisner and Hoory~\cite{eisner2024entropy}.

In this work, we argue that there is a better way to formulate the problem using lifts of graphs. 
Given any finite connected graph $H$ with minimal degree at least two, $n(H,g)$ is the smallest order of a lift of $H$ with girth at least $g$. 
We prove that the Moore bound and the Erd\H{o}s-Sachs construction naturally generalize to the new setting, even in their non-asymptotic form. 
For the sake of avoiding complications, we exclude the case of $H$ being a cycle, as it is well understood.
We also note that the $n(H,g)$ problem is not a generalization of the $n(k,g)$ problem, since the family of all $k$-regular graphs is not necessarily a lift of some base graph $H$.

Following is the main theorem of the paper. The required notation will be introduced later. At this point, it suffices to know that multigraph with loops broadens the class of admissible graphs $H$, and that $\rho(\tH)$ is some average of the degree minus one over the vertices of $H$. If $H$ is $k$-regular, then $\rho(\tH)=k-1$. 

\begin{theorem}\label{theorem:main}
For a finite connected multigraph with loops $H$ with minimal degree at least two and maximal degree greater than two:
\begin{equation}\label{liftAsymp}
\Omega({\rho(\tH)}^{g/2}) \leq n(H,g) \leq O({\rho(\tH)}^g),
\end{equation}
where $\rho(\tH)$ is the growth rate of $H$'s universal covering tree.
\end{theorem}

The rest of the paper is organized as follows.
In Section~\ref{section:lifts} we define multigraph with loops, graph lifts, the universal cover of a graph, its growth rate $\rho(\tH)$ and we show that it is the same as the Perron eigenvalue of non-backtracking adjacency matrix of the graph.
In Section~\ref{section:bounds} we prove the non-asymptotic form of Theorem~\ref{theorem:main} and derive the theorem as a corollary. 
In Section~\ref{section:comparing} we explore the relation of $n(H,g)$ to similar definitions. 
Namely, the order of the smallest graph sharing the same universal cover and the smallest graph with the same proportion of each degree. We demonstrate these results for the base graphs $K_{3,2}$ and $H_{23}$, shown in Figure~\ref{fig:three_graphs}~(a) and (b). 
In Section~\ref{section:concrete}, we give numeric results for upper and lower bounds on $n(H_{23},g)$. Namely, the Moore bound, the Erd\H{o}s-Sachs construction, exhaustive search, and several greedy construction algorithms. 
We conclude our work with several conjectures and open questions.

\section{Lifts of Graphs and the Growth Rate of Balls}\label{section:lifts}
Graphs are undirected. A base graph $H$ is always connected, with minimal degree $\geq 2$ and maximal degree $>2$; its lifts have the same degrees but need not be connected.
Graphs are not necessarily simple, as we allow multiple edges and two types of self loops: half-loops, and whole-loops,
following the convention of Friedman and others~\cite{angel2015non,friedman1993some,friedmanKohler2014relativized,leighton1982finite}. 
If the graph is simple, the extra notation can be ignored, and each edge is regarded as a pair of directed edges, each being the inverse of the other.

A graph $H$ is the tuple $(V,\dirE,h,t,\iota)$, where $V$ is the vertex set, $\dirE$ is the directed edge set, $h,t:\dirE \rightarrow V$ respectively map an edge to its head or tail vertex, and the inverse $\iota$, maps an edge to its inverse $\iota(e)=e^{-1}$.
The obvious relations hold: $h(e^{-1})=t(e)$ and the inverse mapping is an involution. 
An edge $e$ with $h(e)=t(e)$ is called a loop: a half-loop if $e=e^{-1}$ and a whole-loop otherwise.
The undirected edges of $H$ are the orbits $\{e,e^{-1}\}$ of $\iota$, and we denote their set by $E(H)$; thus a half-loop is an undirected edge consisting of a single directed edge.
The degree of $v$, denoted $\deg(v)$ is the number of edges whose tail is $v$, i.e. $\deg(v) = |t^{-1}(v)|=|h^{-1}(v)|$.
Note that while a half-loop contributes one to the degree, a whole-loop contributes two.
The neighborhood of $v$ is $\{h(e):e \in t^{-1}(v)\}$, namely the heads of edges emanating from $v$.
A length $k$ walk $\omega$ from $u$ to $v$ is an edge sequence $e_1,\ldots,e_k$ such that $t(e_1)=u$, $h(e_k)=v$ and $h(e_i)=t(e_{i+1})$ for all $i$.
The walk $\omega$ is non-backtracking if $e_{i+1} \neq (e_i)^{-1}$ for all $i$.
The girth of $H$ is the minimal positive length $k$ of a closed non-backtracking walk, that is $h(e_k)=t(e_1)$.
Note that any graph with a whole-loop or half-loop has girth one.
For brevity, a \emph{girth $g$ graph} is a graph of girth at least $g$.
Let $\dist(u,v)$ be the length of the shortest walk from $u$ to $v$ (necessarily, such a walk is always non-backtracking), or infinity if no such walk exists. 
Denote the radius $r$ ball centered at some vertex $v$ or edge $e$ by $B_r(v) = \{u : \dist(u,v) \leq r\}$ and $B_r(e) = B_r(h(e)) \cup B_r(t(e))$. 
As $B_r(e)=B_r(e^{-1})$, the ball $B_r(e)$ is also well defined for an undirected edge $e \in E(H)$.

Let $H$ be a finite undirected and connected graph, as above.
Following the definition of a topological covering we say that the graph $G$ covers $H$, or that $G$ is a lift of $H$, if there is a surjective homomorphism $\pi:G \rightarrow H$ that is a bijection of the edges emanating from $v$ to the edges emanating from $\pi(v)$ for all $v \in V(G)$, 
i.e. $\pi(t^{-1}(v)) = t^{-1}(\pi(v))$. 
One can verify that fiber size is fixed for all vertices and edges. 
Namely, there is an integer $n$, which we call the height of the lift, such that $|\pi^{-1}(v)|=|\pi^{-1}(e)|=n$ for all $v \in V(H)$ and $e \in \dirE(H)$. 
Thus a lift of height $n$ has order $n\,|V(H)|$. Throughout Sections~\ref{section:lifts} and~\ref{section:bounds}, $n$ denotes the height of a lift, whereas $n(H,g)$ always denotes an order.
Every connected component of a lift of $H$ is itself a lift of $H$, since its image is non-empty and closed under taking neighbors, and its girth is at least that of the lift. In particular, $n(H,g)$ is attained by a connected lift.
We identify $V(G)$ with $V(H) \times [n]$ and $\dirE(G)$ with $\dirE(H) \times [n]$, where $[n] = \{1,2,\ldots,n\}$.
Furthermore, we regard $\pi$ as a mapping $\perm_\pi$ from $\dirE(H)$ to the permutation group $S_n$, 
where for any $e \in \dirE(H)$ and $i \in[n]$ the edge in $\pi^{-1}(e)$ with tail $(t(e),i)$ has its head at $(h(e),\perm_\pi(e)(i))$. It follows that 
$\perm_\pi(e^{-1})=\perm_\pi(e)^{-1}$ and that $\perm_\pi$ of a half-loop is an involution.

The universal cover of $H$, denoted by $\tH$, is the connected graph covering all connected lifts of $H$. This graph, which is unique up to isomorphism, is an infinite tree. One way to construct $\tH$ is to choose some arbitrary vertex $v_0 \in V(H)$ and let $V(\tH)$ be the set of all finite non-backtracking walks from $v_0$. Two vertices of $\tH$ are adjacent if one walk extends the other by a single step and the root of the tree is the empty walk $\epsilon$.

We define the growth rate of $\tH$ by $\rho(\tH) = \lim_{r \rightarrow \infty} |B_r(v)|^{1/r}$ for some arbitrary vertex $v \in V(\tH)$.
The limit exists and is independent of $v$ by Theorem~\ref{theorem:treeBallSize}.
The non-backtracking adjacency operator of $H$ is the 0-1 matrix $B$ of directed edges versus directed edges, where $B_{f,e}=1$ if $h(f)=t(e)$ and $e \neq f^{-1}$.
For a square matrix $M$, its spectral radius $\rho(M)$ is the maximum absolute value of its (complex) eigenvalues.
Then by the following propositions, as $H$ has a minimal degree of $\geq 2$ and maximal degree $>2$, 
the matrix $B$ is non-negative irreducible, so by the Perron-Frobenius Theorem, its spectral radius $\rho(B)$ is an algebraically simple eigenvalue corresponding to strictly positive left and right eigenvectors. Furthermore, all eigenvalues whose absolute value is equal to $\rho(B)$ are algebraically simple.

\begin{proposition}[\cite{glover2021non} proposition 3.3,\cite{eisner2024entropy}]\label{proposition:glover2021non}
    The non-backtracking adjacency matrix $B_H$ is irreducible iff $H$ is connected with $\mindeg(H) \geq 2$ and $\maxdeg(H) > 2$.
\end{proposition}

\begin{proposition}[Perron–Frobenius, Theorem 8.4.4 from~\cite{horn2012matrix}]\label{proposition:perron}
    Given a non-negative irreducible matrix $M \in \MnR$ for $n \ge 2$, the spectral radius $\rho(M)$ is a positive and algebraically simple eigenvalue.
    Furthermore, up to scaling there are unique vectors $\x,\y$ so that $M \x = \rho(M) \x$ and $\y^T M = \rho(M) \y^T$;
    these vectors are strictly positive.
\end{proposition}

\begin{proposition}[Corollary 8.4.6 (c) from~\cite{horn2012matrix}]\label{proposition:max_abs_lambda_simple}
    Given a non-negative irreducible matrix $M \in \MnR$, let $k$ be the number of distinct maximal eigenvalues in absolute value. 
    Then, these eigenvalues are algebraically simple and their values are $\rho(M) e^{2 \pi i j/k}$ for $j=0,\ldots,k-1$.
\end{proposition}

\begin{theorem}\label{theorem:treeBallSize}
    For any finite undirected connected graph $H$ with $\mindeg(H) \geq 2$ and $\maxdeg(H) > 2$, 
    there are positive constants $c_1, c_2$ such that for any vertex $v \in V(\tH)$ and non-negative integer $r$:
    \[
        c_1 \rho(B_H)^r \leq |B_r(v)| \leq c_2 \rho(B_H)^r.
    \]
\end{theorem}

\begin{corollary}\label{corollary:treeGrowthRate}
    For any finite undirected connected graph $H$ with $\mindeg(H) \geq 2$ and $\maxdeg(H) > 2$,
    the growth rate of its universal cover $\tH$ is equal to the spectral radius of $B_H$: 
    $$\rho(\tH) = \rho(B_H).$$    
\end{corollary}

The fact that $\rho(\tH)$ and $\rho(B_H)$ are the same, allows us to simplify notation and denote this number by $\rho(H)$, or even just $\rho$ if the underlying graph is clear from the context.

We provide two distinct proofs for Theorem~\ref{theorem:treeBallSize}. 
The first is short and self-contained, using only the Perron–Frobenius theorem.
The second, which is longer, yields more: it employs Proposition~\ref{proposition:max_abs_lambda_simple} to show that $|B_r(v)|/\rho(B_H)^r$ converges for each congruence class of $r$ modulo the number $k$ of eigenvalues of $B_H$ with maximal absolute value. 
In particular, Theorem~\ref{theorem:treeBallSize} cannot be strengthened to $|B_r(v)| \sim c \rho(B_H)^r$; see Remark~\ref{remark:periodic_ball} and the example $K_{3,2}$, where $k=4$.

\begin{proof}[First Proof of Theorem~\ref{theorem:treeBallSize}]
    Let $H$ be a graph satisfying the hypotheses, let $\tH$ be its universal cover, and let $\pi$ be the respective cover map.
    The non-backtracking adjacency matrix $B_H$ is non-negative irreducible by Proposition~\ref{proposition:glover2021non}, 
    implying that $\rho = \rho(B_H)$ is an eigenvalue of $B_H$ with strictly positive left eigenvector $\y$, satisfying $\y^T B_H = \rho \, \y^T$.
    Let $v$ be an arbitrary vertex in $\tH$.
    Then, the size of the radius $r$ ball centered at $v$, is
    \begin{equation}\label{eqn:tree_ball_size}
        |B_r(v)|  = 1 + \sum_{i=1}^r \Delta_i,
    \end{equation}
    where $\Delta_i=\Delta_i(v)$ denotes the number of length $i$ non-backtracking walks from $v$.
    The quantity $\Delta_i$ can be expressed using the non-backtracking matrix $B_H$ as
    \begin{equation}\label{eqn:tree_ball_size2}
        \Delta_i = \delta_v^T (B_H)^{i-1} \one,
    \end{equation}
    where $\one$ is the all-one vector indexed by the set of directed edges $\dirE$, and $\delta_v$ is the indicator vector on the edges emanating from $\pi(v)$, where the entry $(\delta_v)_e$ is one if $t(e) = \pi(v)$ and zero otherwise.
    This yields
    \begin{equation*}
        |B_r(v)|  = 1 + \sum_{i=1}^r \deltav^T (B_H)^{i-1} \one.
    \end{equation*}
    Also, as $B_H$ is non-negative irreducible, all edges in $\dirE$ are reachable from $v$, 
    implying that for some integer $r_0$ the following vector is strictly positive:
    \begin{equation*}
        \z^T = \sum_{i=1}^{r_0} \deltav^T (B_H)^{i-1}.
    \end{equation*}
    It follows that there is a constant $c_1>0$ such that $c_1 \y \leq \z$ and therefore for all $r \geq r_0$ we can lower-bound the size of the radius $r$ ball centered at $v$ by
    \begin{equation*}
        |B_r(v)| \geq \sum_{i=r-r_0+1}^r \deltav^T (B_H)^{i-1} \one = \z^T (B_H)^{r-r_0} \one \geq c_1 \rho^{r-r_0} \y^T \one \geq \Omega(\rho^r).
    \end{equation*}

    As for the upper bound, as $\y$ is positive, there is a constant $c_2 >0$ so that $\deltav \leq c_2 \y$, and therefore
    \begin{equation*}
        |B_r(v)| \leq 1 + c_2 \sum_{i=1}^r \y^T (B_H)^{i-1} \one 
                  = 1 + c_2 (\y^T \one) \sum_{i=1}^r \rho^{i-1}.
    \end{equation*}
    By the hypotheses on $H$ we have $\rho=\rho(B_H) > 1$, implying that the geometric sum is $O(\rho^r)$.
\end{proof}

\begin{proof}[Second Proof of Theorem~\ref{theorem:treeBallSize}]
    As in the first proof, let $H$ be a graph satisfying the hypotheses, let $\tH$ be its universal cover, let $\pi$ be the respective cover map and let $v$ be an arbitrary vertex in $\tH$.
    Let $B_H = V^{-1} \left[\bigoplus_{j=1}^m J_j\right] V$ be the Jordan decomposition of $B_H$.
    Since $B_H$ is non-negative and irreducible, by Proposition~\ref{proposition:glover2021non},  
    it follows by Proposition~\ref{proposition:max_abs_lambda_simple} that the maximum-modulus eigenvalues correspond to $k$ scalar $1 \times 1$ Jordan blocks.
    Thus, the decomposition can be rearranged so that the Jordan blocks $J_p$ for $p=1,\ldots,k$ are the scalars $J_p = (\rho \, \omega_k^{p-1})$, where
    $\rho = \rho(B_H)$ and $\omega_k = e^{2 \pi \mathrm{i}/k}$ (here $\mathrm{i}$ denotes $\sqrt{-1}$).    
    
    Moreover, the left and right Perron eigenvectors of $B_H$, satisfying $\y^T B_H = \rho \y^T$ and $B_H \x = \rho \x$, are the first row of $V$ and first column of $V^{-1}$, 
    i.e. $\y^T = (V)_{1,\cdot}$ and $\x = (V^{-1})_{\cdot,1}$.
    
    Denoting the zero padded $J_j$ blocks by $\widehat{J_j}$, we write $B_H$ as
    \begin{equation*}
        B_H = V^{-1} \left[\sum_{j=1}^m \widehat{J_j}\right] V = V^{-1} \left[\sum_{p=1}^k E_p \rho \, \omega_k^{p-1} + R\right] V,
    \end{equation*}
    where the matrix $E_p$ is all-zero except for a one in its $p$-th diagonal entry, and the residue matrix $R=\sum_{j=k+1}^m \widehat{J_j}$ satisfies the condition $\rho(R) < \rho$.

    Applying the decomposition to the ball size expression given by \eqref{eqn:tree_ball_size} and \eqref{eqn:tree_ball_size2}, yields:
    \begin{align*}
        \Delta_i = \delta_v^T (B_H)^{i-1} \one 
        &= \rho^{i-1} \sum_{p=1}^k (\delta_v^T V^{-1} E_p V \one) \, \omega_k^{(p-1)(i-1)} + \delta_v^T V^{-1} R^{i-1} V \one \\
        &= \rho^{i-1} \sum_{p=1}^k a_p \, \omega_k^{(p-1)(i-1)} + r_i,
    \end{align*}
    where $a_p = \delta_v^T V^{-1} E_p V \one$ and the residue term is $r_i = \delta_v^T V^{-1} R^{i-1} V \one$.
    
    Letting $\|\cdot\|$ denote the Euclidean norm and the operator norm it induces, we bound the residue term $r_i$ by:
    \begin{equation*}
        |r_i| 
        \leq \|\delta_v\| \cdot \|V^{-1}\|  \cdot \|R^{i-1}\|  \cdot \|V\|  \cdot \|\one\| 
        = o(\rho^i) \,\,\text{ as }i \rightarrow \infty,
    \end{equation*}
    where the last equality holds because $\rho(R) < \rho$ and by Gelfand's formula $\lim_m \|R^m\|^{1/m} = \rho(R)$. 
    Therefore, fixing $\alpha$ with $\rho(R) < \alpha < \rho$ implies that there is a constant $C>0$ so that $\|R^m\| \le C \alpha^m$ for all $m$ and therefore that $\|R^m\| = o(\rho^m)$.
    
    Furthermore, $a_p$ are complex constants, and the coefficient $a_1$, corresponding to the Perron eigenvalue $\rho$, is positive: 
    \[
        a_1 = \delta_v^T V^{-1} E_1 V \one = \delta_v^T (V^{-1} E_1) (E_1 V) \one = \delta_v^T \x \y^T \one > 0.
    \]
    
    Now, let $i-1 \equiv s \pmod{k}$, for $s \in \{0, \ldots, k-1\}$. 
    Then, we can rewrite $\Delta_i$ as :
    \begin{equation}\label{eqn:periodic_ball_size}
        \Delta_i = \rho^{i-1} \sum_{p=1}^k a_p \, \omega_k^{(p-1)s} + r_i = \rho^{i-1} \widehat{a_{s}} + r_i,
    \end{equation}
    where $\widehat{a_{s}} = \sum_{p=1}^k a_p \, \omega_k^{(p-1)s}$. 
    The sequence $\{\widehat{a_{s}}\}_{s=0,\ldots,k-1}$ is not all zero, because the coefficients $\{a_p\}_{p=1 \ldots k}$ are not all zero and the matrix $\{\omega_k^{(p-1)s}\}_{p,s}$ is non-singular, being a Vandermonde matrix (specifically, the discrete Fourier transform matrix). 
    
    Since $|r_i|=o(\rho^i)$, and since there exists at least one value $s_0$ for which $\widehat{a_{s_0}} \neq 0$, 
    it follows that $\Delta_i = \Theta(\rho^i)$ for all $i$ with $i-1 \equiv s_0 \pmod{k}$. 
    Furthermore, as $\mindeg(H) \ge 2$, every non-backtracking walk can be extended, so $\Delta_i$ is a non-decreasing function of $i$. 
    Therefore, $\Delta_i = \Theta(\rho^i)$ for all $i$. 
    
    Finally, as the assumptions on $H$ imply that $\rho > 1$, the sum in \eqref{eqn:tree_ball_size} is asymptotically dominated by its last term, $\Delta_r$, implying that $|B_r(v)|=\Theta(\rho(B_H)^r)$ as $r \rightarrow \infty$, as claimed. 
\end{proof}

\begin{remark}\label{remark:periodic_ball}
    Equation~\eqref{eqn:periodic_ball_size} gives more than Theorem~\ref{theorem:treeBallSize}
    requires: summing it over $i \le r$ shows that $|B_r(v)|/\rho^r$ converges for each congruence
    class of $r$ modulo $k$, but the $k$ limits may differ, in which case the ratio itself does not
    converge. This is why the theorem is stated with two constants and not as
    $|B_r(v)| \sim c\,\rho^r$. Periods $k>1$ are easy to come by, as replacing every edge of a graph
    by a path of length $\ell$ multiplies its period by $\ell$. 
    For instance $K_{3,2}$, depicted in Figure~\ref{fig:three_graphs}(a), has $k=4$ with maximal modulus eigenvalues $\pm\sqrt2, \pm i\sqrt2$. 
    Counting the non-backtracking walks in its universal cover, from a degree three vertex $v$, 
    gives $\Delta_i = 3 \cdot 2^{\lfloor (i-1)/2 \rfloor}$, so $|B_r(v)|$ equals $6 \cdot 2^{r/2}-5$ for even $r$ and
    $9 \cdot 2^{(r-1)/2}-5$ for odd $r$, and $|B_r(v)|/\rho^r$ has the two limit points $6$ and
    $9/\sqrt2$.
\end{remark}

\paragraph{Lifts as permutation families.}
We conclude this section with two simple facts on lifts that are used in Section~\ref{section:bounds}.
Recall that a lift $G$ of $H$ of height $n$ is described by the permutations $\sigma_f = \perm_\pi(f) \in S_n$ for $f \in \dirE(H)$:
the edge $(f,i) \in \dirE(G)$ has tail $(t(f),i)$, head $(h(f),\sigma_f(i))$ and inverse $(f^{-1},\sigma_f(i))$, where $\sigma_{f^{-1}}=\sigma_f^{-1}$.
We call $L_i = V(H) \times \{i\}$ the $i$-th layer of $G$.

\begin{lemma}\label{lemma:permLift}
    Let $\{\sigma_f\}_{f\in\dirE(H)}$ be permutations of $[n]$ satisfying $\sigma_{f^{-1}}=\sigma_f^{-1}$ for all $f \in \dirE(H)$.
    Then the vertex set $V(H) \times [n]$, the directed edge set $\dirE(H) \times [n]$, and the maps
    \[
        t(f,i) = (t(f),i), \qquad h(f,i) = (h(f),\sigma_f(i)), \qquad (f,i)^{-1} = (f^{-1},\sigma_f(i))
    \]
    define a lift of $H$ with cover map $(x,i)\mapsto x$, in which the half-loops are the edges $(f,i)$ with $f=f^{-1}$ and $\sigma_f(i)=i$.
\end{lemma}

\begin{proof}
    The inverse map is an involution, since $\big((f,i)^{-1}\big)^{-1} = \big(f, \sigma_{f^{-1}}(\sigma_f(i))\big) = (f,i)$.
    It also satisfies $h\big((f,i)^{-1}\big) = \big(t(f),\sigma_{f^{-1}}(\sigma_f(i))\big) = t(f,i)$.
    The claim on half-loops is immediate from the definition of the inverse.
    The projection $(x,i)\mapsto x$ is a homomorphism mapping the edges $\{(f,i) : f \in t^{-1}(x)\}$ emanating from $(x,i)$ bijectively onto $t^{-1}(x)$, and it is surjective as $n \geq 1$.
\end{proof}

\begin{lemma}\label{lemma:treeLayers}
    Let $G$ be a lift of $H$ of height $n$ and let $T$ be a spanning tree of $H$.
    Then the vertices in each fiber of $G$ can be relabeled so that $\sigma_f = \mathrm{id}$ for all $f \in \dirE(T)$.
    In such a labeling:
    \begin{itemize}
        \item every layer $L_i$ contains the copy $\{(f,i) : f \in \dirE(T)\}$ of $T$, so any two vertices in the same layer are at distance at most $\diam(T)$ in $G$;
        \item every edge of $G$ joining two different layers covers an edge in $\dirE(H)\setminus\dirE(T)$;
        \item the property $\sigma_f|_{\dirE(T)} = \mathrm{id}$ is preserved by permuting the layers, i.e., by relabeling every vertex $(x,i)$ as $(x,\alpha(i))$ for a fixed $\alpha \in S_n$.
    \end{itemize}
\end{lemma}

\begin{proof}
    For bijections $\beta_x : [n] \to [n]$, $x \in V(H)$, relabeling every vertex $(x,i)$ as $(x,\beta_x(i))$ yields the permutations $\beta_{h(f)}\, \sigma_f\, \beta_{t(f)}^{-1}$.
    Root $T$ at some vertex $x_0$ and set $\beta_{x_0}=\mathrm{id}$.
    Then, recursively, set $\beta_{h(f)} = \beta_{t(f)}\,\sigma_f^{-1}$ for every edge $f \in \dirE(T)$ directed away from $x_0$.
    This makes the new permutation of every such $f$ the identity, and hence also that of $f^{-1}$.
    The three listed properties follow directly, the last because $\alpha\,\mathrm{id}\,\alpha^{-1}=\mathrm{id}$.
\end{proof}

\section{Moore and Erd\H{o}s-Sachs Bounds for Graph Lifts}
\label{section:bounds}

In this section we prove the non-asymptotic form of Theorem~\ref{theorem:main} by means of Lemma~\ref{lemma:mooreLift} for the lower bound and Corollary~\ref{corollary:ESlift} for the upper bound. 
The asymptotic result of Theorem~\ref{theorem:main} follows by combining the non-asymptotic results with Theorem~\ref{theorem:treeBallSize}, stating that the size of the radius $r$ ball in the universal cover grows as $\Theta(\rho^r)$.

\begin{lemma}[Moore bound for lifts]\label{lemma:mooreLift}
    Let $H$ be a finite undirected connected graph of minimal degree at least two. Then:
    \begin{itemize}
        \item $n(H,g) \geq \max_{v \in V(\tH)} |B_{\lfloor g/2 \rfloor}(v)|$ if $g$ is odd
        \item $n(H,g) \geq \max_{e \in E(\tH)} |B_{\lfloor g/2 \rfloor - 1}(e)|$ if $g$ is even
    \end{itemize}
\end{lemma}

\begin{proof}
    Let $G$ be a connected lift of $H$ with girth at least $g$, and let $\phi : \tH \to G$ be a cover map.
    As $\phi$ is a bijection on the edges emanating from each vertex, it maps non-backtracking walks to non-backtracking walks.
    Hence, if $x \neq y$ are vertices of $\tH$ with $\dist(x,y) < g$ and $\phi(x) = \phi(y)$, 
    then the image of the path from $x$ to $y$ is a closed non-backtracking walk in $G$ of length less than $g$, 
    which is a contradiction.
    Therefore, $\phi$ is injective on any vertex set of diameter less than $g$.
    The claim follows, as $B_r(v)$ has diameter at most $2r < g$ for $g = 2r+1$, 
    and $B_{r-1}(e)$ has diameter at most $2r-1 < g$ for $g = 2r$.
\end{proof}

\begin{lemma}[Erd\H{o}s-Sachs bound for lifts]\label{lemma:ESlift}
    Let $H$ be a finite undirected connected graph of minimal degree at least two and let $T$ be a spanning tree of $H$.
    Then for any integer $g \geq 2\diam(T)+2$ there exists a connected graph $G$ covering $H$ with $\girth(G) \geq g$ and $\diam(G) \leq g+2\diam(T)$.
\end{lemma}

\begin{corollary}\label{corollary:ESlift}
    Let $H$ be a finite undirected connected graph of minimal degree at least two, let $T$ be some spanning tree of $H$.
    Then for any integer $g \geq 2\diam(T)+2$:
    \begin{equation*}
        n(H,g) \leq \min_{v \in V(\tH)} |B_{g\,+\,2\diam(T)}(v)|.
    \end{equation*}
\end{corollary}

\begin{proof}[Proof of Lemma~\ref{lemma:ESlift}]
    Throughout the proof we write $d=\diam(T)$ and $D_0 = g+2d$, and recall that $g \geq 2d+2$.
    By Lemma~\ref{lemma:liftExists} there exist connected lifts of $H$ with girth at least $g$.
    Let $G$ be one with the minimal number of vertices, and let $n$ be its height.
    We claim that $\diam(G) \leq D_0$.
    Assume to the contrary that there are $u',v' \in V(G)$ with $\dist_G(u',v') > D_0$.
    We derive a contradiction by constructing a lift $G''$ of $H$ with girth at least $g$ and fewer vertices than $G$.
    
    By Lemma~\ref{lemma:treeLayers}, we may assume that $\sigma_f = \mathrm{id}$ for all $f \in \dirE(T)$.
    As $\dist_G(u',v') > D_0 > d$, the vertices $u'$ and $v'$ lie in different layers.
    Hence, after permuting the layers, we may assume that $v' \in L_n$ and $u' \in L_{n-1}$.
    Let $L_{[j]} = L_1 \cup \cdots \cup L_j$.
    We say that an edge of $G$ is \emph{low} if both its endpoints lie in $L_{[n-2]}$, \emph{high} if both lie in $L_{n-1}\cup L_n$, and \emph{crossing} otherwise.
    An edge and its inverse are of the same kind.
    By the following claim, the only high edges are edges covering tree edges.

    \begin{claim}\label{claim_1}
        For every $f \in \dirE(H) \setminus \dirE(T)$, the permutation $\sigma_f$ maps $\{n-1,n\}$ into $[n-2]$. In particular $n \geq 3$.
    \end{claim}
    
    \begin{claimproof}
        Let $i \in \{n-1,n\}$, let $j = \sigma_f(i)$, and assume to the contrary that $j \in \{n-1,n\}$.
        If $j=i$, then the edge $(f,i)$ goes from $(t(f),i)$ to $(h(f),i)$.
        Appending the path in the copy of $T$ in $L_i$ from $(h(f),i)$ to $(t(f),i)$ yields a closed walk of length at most $1+d$.
        This walk is non-backtracking, since $(f,i)$ covers a non-tree edge whereas the path consists of edges covering tree edges (if $h(f)=t(f)$, the path is empty and the walk is the single loop $(f,i)$).
        Hence $\girth(G) \leq 1+d < g$, a contradiction.
        Otherwise $\{i,j\}=\{n-1,n\}$, so $(f,i)$ joins $L_{n-1}$ and $L_n$, implying $\dist_G(v',u') \leq d+1+d < D_0$, again a contradiction.
        Finally, $\dirE(H) \setminus \dirE(T)$ is non-empty, since $|\dirE(H)| = \sum_x \deg(x) \geq 2|V(H)| > |\dirE(T)|$, so the claim forces $n-2 \geq 1$.
    \end{claimproof}
    
    \paragraph{The smaller lift $G''$.}
    Let $\tau \in S_n$ be the transposition exchanging $n-1$ and $n$.
    For every $f \in \dirE(H)$ define $\sigma''_f : [n-2] \to [n-2]$ by
    \begin{equation}\label{eqn:sigma_new}
        \sigma''_f(j) =
        \begin{cases}
            \sigma_f(j) & \text{if } \sigma_f(j) \leq n-2, \\
            \sigma_f\big(\tau(\sigma_f(j))\big) & \text{if } \sigma_f(j) \in \{n-1,n\}.
        \end{cases}
    \end{equation}
    This is well defined: the second case occurs only for $f \notin \dirE(T)$, and then $\tau(\sigma_f(j)) \in \{n-1,n\}$ is mapped by $\sigma_f$ into $[n-2]$ by Claim~\ref{claim_1}.
    
    \begin{claim}\label{claim_2}
        For every $f \in \dirE(H)$, $\sigma''_{f^{-1}} \circ \sigma''_f = \mathrm{id}_{[n-2]}$.
    Consequently, $\sigma''_f$ is a permutation of $[n-2]$ with $\sigma''_{f^{-1}} = (\sigma''_f)^{-1}$; it is an involution if $f$ is a half-loop; and $\sigma''_f=\mathrm{id}$ if $f \in \dirE(T)$.
    \end{claim}
    
    \begin{claimproof}
        Let $j \in [n-2]$ and $k=\sigma''_f(j)$.
        If $\sigma_f(j) \leq n-2$, then $k = \sigma_f(j)$ and $\sigma_{f^{-1}}(k) = j \leq n-2$, so $\sigma''_{f^{-1}}(k)=j$.
        Otherwise $\sigma_f(j) = i \in \{n-1,n\}$ and $k = \sigma_f(\tau(i))$, so that $\sigma_{f^{-1}}(k) = \tau(i) \in \{n-1,n\}$, and therefore
        $\sigma''_{f^{-1}}(k) = \sigma_{f^{-1}}(\tau(\tau(i))) = \sigma_f^{-1}(i) = j$.
        Since $[n-2]$ is finite, this left inverse makes $\sigma''_f$ a bijection with inverse $\sigma''_{f^{-1}}$. 
        If $f$ is a half-loop, then $f=f^{-1}$, so $\sigma''_f=(\sigma''_f)^{-1}$ is an involution.
        Finally, if $f\in\dirE(T)$, then $\sigma''_f=\sigma_f=\mathrm{id}$.
    \end{claimproof}

    By Claim~\ref{claim_2} and Lemma~\ref{lemma:permLift}, the family $\{\sigma''_f\}_{f \in \dirE(H)}$ defines a lift $G''$ of $H$ on the vertex set $L_{[n-2]}$, with fewer vertices than $G$.
    We call an edge $(f,j) \in \dirE(G'')$ \emph{old} if $\sigma_f(j) \leq n-2$ and \emph{new} otherwise.
    An old edge has the same tail, head and inverse in $G''$ as in $G$, so the old edges of $G''$ are exactly the low edges of $G$.
    Hence a walk consisting of old edges is non-backtracking in $G''$ if and only if it is non-backtracking in $G$.
    For a new edge $e = (f,j)$ of $G''$ let $\tp(e) = \sigma_f(j) \in \{n-1,n\}$, and define two crossing edges of $G$, both oriented away from $e$:
    \begin{itemize}
        \item $p_e = (f,j)$, going from $t(e)$ to $(h(f), \tp(e)) \in L_{\tp(e)}$, and
        \item $q_e = (f^{-1},\sigma''_f(j))$, going from $h(e)$ to $(t(f), \tau(\tp(e))) \in L_{\tau(\tp(e))}$.
    \end{itemize}
    The head of $q_e$ is computed by $\sigma_{f^{-1}}(\sigma''_f(j)) = \sigma_f^{-1}\big(\sigma_f(\tau(\tp(e)))\big) = \tau(\tp(e))$.
    By the computation in the proof of Claim~\ref{claim_2}, the inverse $e^{-1} = (f^{-1},\sigma''_f(j))$ is a new edge with $\tp(e^{-1}) = \tau(\tp(e))$ and $p_{e^{-1}} = q_e$.
    Since $p_{e'}$ carries the same label as $e'$, it follows that for any two new edges $e,e'$:
    \begin{equation}\label{eqn:witness_inverse}
        q_e = p_{e'} \iff e' = e^{-1}.
    \end{equation}

    \paragraph{The layer-deletion view.}
    The definition~\eqref{eqn:sigma_new} has a simple description, illustrated in Figure~\ref{fig:replace_new}.
    Deleting the layers $L_{n-1}$ and $L_n$ from $G$ leaves the induced subgraph $G[L_{[n-2]}]$, whose edges are the old edges of $G''$.
    Fix $f \in \dirE(H)\setminus\dirE(T)$. By Claim~\ref{claim_1}, applied to $f$ and to $f^{-1}$, the deletion cuts four edges of $G$ covering $f$:
    the two edges $(f,\sigma_f^{-1}(n-1))$ and $(f,\sigma_f^{-1}(n))$ going up from $L_{[n-2]}$ into the top layers,
    which leave two vertices of $L_{[n-2]}$ without an outgoing edge covering $f$,
    and the two edges $(f,n-1)$ and $(f,n)$ coming down from the top layers into $L_{[n-2]}$,
    which leave two vertices without an incoming edge covering $f$.
    Accordingly, there are exactly two new edges covering $f$, namely $e=(f,j)$ and $\hat e=(f,\hat\jmath)$ with $j=\sigma_f^{-1}(n-1)$ and $\hat\jmath=\sigma_f^{-1}(n)$, so that $\tp(e)=n-1$ and $\tp(\hat e)=n$.
    By~\eqref{eqn:sigma_new}, $e$ goes from the tail of the upward edge $(f,j)$ to the head of the downward edge $(f,n)$,
    and $\hat e$ from the tail of $(f,\hat\jmath)$ to the head of $(f,n-1)$.
    Thus $e$ and $\hat e$ reconnect the four cut edges crosswise, and each deficiency is closed exactly once.
    If $f$ is a half-loop, then $\hat e=e^{-1}$; see Figure~\ref{fig:replace_new}(b).
    
    \begin{figure}[ht]
      \centering
      \subfloat[]{%
      \begin{tikzpicture}[scale=0.85,
          newedgeB/.style={newedge, blue!35, line width=0.8pt},
          crossedgeB/.style={crossedge, red!35, line width=0.7pt}]
        \draw[layer]     (-2.1,3.65) rectangle (5.6,4.35);
        \draw[layer]     (-2.1,2.65) rectangle (5.6,3.35);
        \draw[lowregion] (-2.1,-0.1) rectangle (5.6,2.25);
        \node[lbl, anchor=west] at (5.65,4)   {$L_n$};
        \node[lbl, anchor=west] at (5.65,3)   {$L_{n-1}$};
        \node[lbl, anchor=west] at (5.65,1.1) {$L_{[n-2]}$};
        \draw[fiber] (-0.25,0.2) rectangle (0.25,4.2);
        \draw[fiber] (3.35,0.2) rectangle (3.85,4.2);
        \node[lbl] at (0,-0.45)   {$\pi^{-1}(x)$};
        \node[lbl] at (3.6,-0.45) {$\pi^{-1}(y)$};
        \node[vtx] (A)  at (0,0.45)   {};   
        \node[vtx] (A2) at (0,1.55)   {};   
        \node[vtx] (D2) at (0,3)      {};   
        \node[vtx] (D)  at (0,4)      {};   
        \node[vtx] (B)  at (3.6,0.75) {};   
        \node[vtx] (B2) at (3.6,1.85) {};   
        \node[vtx] (C)  at (3.6,3)    {};   
        \node[vtx] (C2) at (3.6,4)    {};   
        \node[lbl, anchor=east] at ([xshift=-8pt]A)  {$(x,j)$};
        \node[lbl, anchor=east] at ([xshift=-8pt]A2) {$(x,\hat\jmath)$};
        \node[lbl, anchor=east] at ([xshift=-8pt]D2) {$(x,n-1)$};
        \node[lbl, anchor=east] at ([xshift=-8pt]D)  {$(x,n)$};
        \node[lbl, anchor=west] at ([xshift=8pt]B)   {$(y,\sigma''_f(j))$};
        \node[lbl, anchor=west] at ([xshift=8pt]B2)  {$(y,\sigma''_f(\hat\jmath))$};
        \node[lbl, anchor=west] at ([xshift=8pt]C)   {$(y,n-1)$};
        \node[lbl, anchor=west] at ([xshift=8pt]C2)  {$(y,n)$};
        \draw[newedgeB]   (A2) -- node[lbl, pos=0.42, above=1pt] {$\hat e$} (B2);
        \draw[crossedgeB] (A2) -- (C2);
        \draw[crossedgeB] (D2) -- (B2);
        \draw[newedge]   (A) -- node[lbl, pos=0.5, below=0pt] {$e=(f,j)$} (B);
        \draw[crossedge] (A) -- node[lbl, pos=0.86, above left=-1pt] {$p_e$} (C);
        \draw[crossedge] (D) -- node[lbl, pos=0.1, above right=-1pt] {$q_e^{-1}$} (B);
      \end{tikzpicture}}      \hspace{1.2em}
      \subfloat[]{%
      \begin{tikzpicture}[scale=0.85]
        \draw[layer]     (-2.5,3.65) rectangle (2.8,4.35);
        \draw[layer]     (-2.5,2.65) rectangle (2.8,3.35);
        \draw[lowregion] (-2.5,-0.1) rectangle (2.8,2.25);
        \node[lbl, anchor=west] at (2.8,4)   {$L_n$};
        \node[lbl, anchor=west] at (2.8,3)   {$L_{n-1}$};
        \node[lbl, anchor=west] at (2.8,1.1) {$L_{[n-2]}$};
        \draw[fiber] (-0.5,0.2) rectangle (0.5,4.2);
        \node[lbl] at (0,-0.45) {$\pi^{-1}(x)$};
        \node[vtx] (A) at (-0.35,0.5) {};
        \node[vtx] (B) at (0.35,1.7)  {};
        \node[vtx] (C) at (-0.35,3)   {};
        \node[vtx] (D) at (0.35,4)    {};
        \node[lbl, anchor=east] at ([xshift=-6pt]A) {$(x,j)$};
        \node[lbl, anchor=west] at ([xshift=6pt]B)  {$(x,\hat\jmath)$};
        \node[lbl, anchor=east] at ([xshift=-6pt]C) {$(x,n-1)$};
        \node[lbl, anchor=west] at ([xshift=6pt]D)  {$(x,n)$};
        \draw[newedge]   (A) -- node[lbl, pos=0.3, right=7pt] {$e,\ \hat e=e^{-1}$} (B);
        \draw[crossedge] (A) -- node[lbl, pos=0.45, left=6pt] {$p_e=q_{\hat e}$} (C);
        \draw[crossedge] (D) -- node[lbl, pos=0.45, right=6pt] {$q_e^{-1}=p_{\hat e}^{-1}$} (B);
      \end{tikzpicture}}
      \caption{The new edges covering $f$ (solid) and the edges of $G$ covering $f$ cut by deleting $L_{n-1}$ and $L_n$ (dashed).
             The new edge $e=(f,j)$, with $\tp(e)=n-1$, goes from the tail of $p_e$ to the head of $q_e^{-1}$.
             (a) $f$ is not a half-loop; $e$ and $\hat e=(f,\hat\jmath)$ (pale) reconnect the four cut edges crosswise.
             (b) $f$ is a half-loop at $x$; then $\hat e=e^{-1}$, and the four cut edges are two edges and their inverses.}
         \label{fig:replace_new}
    \end{figure}
    
    \begin{claim}\label{claim_3}
        $\girth(G'') \geq g$.
    \end{claim}
    
    \begin{claimproof}
        Assume to the contrary that $G''$ has a closed non-backtracking walk of positive length less than $g$, and let $C = e_1,\ldots,e_k$ be such a walk with minimal $k$.
        
        First, $C$ is cyclically non-backtracking, i.e., also $e_1 \neq e_k^{-1}$, and all indices below are taken modulo $k$.
        Indeed, $k=2$ and $e_2=e_1^{-1}$ is impossible since $C$ is non-backtracking, and if $k \geq 3$ and $e_k = e_1^{-1}$ then $e_2,\ldots,e_{k-1}$ is a shorter closed non-backtracking walk.
        If $k=1$, it suffices to show that $G''$ has no half-loops.
        An old half-loop of $G''$ is a half-loop of $G$, contradicting $\girth(G) \geq g \geq 2$.
        A new edge $(f,j)$ with $f$ a half-loop is not a half-loop of $G''$, since $\sigma''_f(j)=j$ would give $\tau(\sigma_f(j)) = \sigma_f^{-1}(j) = \sigma_f(j)$, while $\tau$ has no fixed point in $\{n-1,n\}$.
        
        Next, $C$ contains at least one new edge, since otherwise it is a closed non-backtracking walk in $G$ of length less than $g$.
        For a new edge $e_a$ of $C$, let $e_b$ be the next new edge along $C$ (possibly $e_b = e_a$ if $C$ only has a single new edge), and let $P_a = e_{a+1},\ldots,e_{b-1}$ be the possibly empty segment of old edges between them.
        Then $P_a$ is a non-backtracking walk in $G$ from $h(e_a)$ to $t(e_b)$; this uses the cyclic non-backtracking property of $C$ when $P_a$ passes through the junction of $e_k$ and $e_1$. 
        Moreover, the segments $P_a$ for different new edges $e_a$ are disjoint parts of $C$.
        In $G$, the edge $q_{e_a}$ joins $h(e_a)$ to $L_{\tau(\tp(e_a))}$ and the edge $p_{e_b}$ joins $t(e_b)$ to $L_{\tp(e_b)}$.
        We consider two cases, illustrated in Figure~\ref{fig:claim3}.
        
        \begin{figure}[ht]
          \centering
          \subfloat[]{%
          \begin{tikzpicture}[xscale=0.6, yscale=0.68]
            \draw[layer]     (-2.0,4.25) rectangle (8.0,4.95);
            \draw[layer]     (-2.0,3.25) rectangle (8.0,3.95);
            \draw[lowregion] (-2.0,-1.35) rectangle (8.0,2.3);
            \node[lbl, anchor=west] at (8.05,4.6) {$L_{\tau(m)}$};
            \node[lbl, anchor=west] at (8.05,3.6) {$L_m$};
            \node[lbl, anchor=west] at (8.05,0.5) {$L_{[n-2]}$};
            \node[vtx] (V)  at (-1.4,4.6) {};
            \node[vtx] (Q)  at (1.0,4.6)  {};
            \node[vtx] (Ta) at (0,0.6)    {};
            \node[vtx] (Ha) at (1.5,1.6)  {};
            \node[vtx] (Tb) at (4.5,1.6)  {};
            \node[vtx] (Hb) at (6.0,0.6)  {};
            \node[vtx] (R)  at (5.0,3.6)  {};
            \node[vtx] (U)  at (7.4,3.6)  {};
            \node[lbl, anchor=south west] at ([xshift=-22pt, yshift=2pt]V) {$u'$ or $v'$};
            \node[lbl, anchor=north east] at ([xshift=28pt, yshift=-1pt]U) {$u'$ or $v'$};
            \node[vtx] (X1) at (4.3,-0.8) {};
            \node[vtx] (X2) at (1.7,-0.8) {};
            \draw[restC] (Hb) to[out=-60, in=0] (X1);
            \draw[newedge] (X1) -- (X2);
            \draw[restC] (X2) to[out=180, in=-120] (Ta);
            \node[lbl, gray!70!black] at (6.5,-0.35) {$C$};
            \draw[treepath]  (V) -- node[lbl, below=1.5pt] {$\le d$} (Q);
            \draw[crossedge] (Q) -- node[lbl, right] {$q_{e_a}^{-1}$} (Ha);
            \draw[oldpath]   (Ha) to[out=30, in=150] node[lbl, below=10pt] {$|P_a| \le k-1$} (Tb);
            \draw[crossedge] (Tb) -- node[lbl, right] {$p_{e_b}$} (R);
            \draw[treepath]  (R) -- node[lbl, above=1.5pt] {$\le d$} (U);
            \draw[newedge] (Ta) -- node[lbl, above left] {$e_a$} (Ha);
            \draw[newedge] (Tb) -- node[lbl, above right] {$e_b$} (Hb);
          \end{tikzpicture}}          \hfill
          \subfloat[]{%
          \begin{tikzpicture}[xscale=0.6, yscale=0.68]
            \draw[layer]     (-2.0,4.25) rectangle (8.0,4.95);
            \draw[layer]     (-2.0,3.25) rectangle (8.0,3.95);
            \draw[lowregion] (-2.0,-1.35) rectangle (8.0,2.3);
            \node[lbl, anchor=west] at (8.05,4.6) {$L_{\tau(m)}$};
            \node[lbl, anchor=west] at (8.05,3.6) {$L_m$};
            \node[lbl, anchor=west] at (8.05,0.5) {$L_{[n-2]}$};
            \node[vtx] (Q)  at (0.9,3.6)  {};
            \node[vtx] (R)  at (4.3,3.6)  {};
            \node[vtx] (Ta) at (-0.2,0.6) {};
            \node[vtx] (Ha) at (1.3,1.6)  {};
            \node[vtx] (Tb) at (3.9,1.6)  {};
            \node[vtx] (Hb) at (5.4,0.6)  {};
            \fill[orange!15] (Ha.center) to[out=30, in=150] (Tb.center) -- (R.center) -- (Q.center) -- cycle;
            \node[lbl] at (2.6,2.75) {$W$};            
            \node[vtx] (X1) at (4.6,-0.85) {};
            \node[vtx] (X2) at (3.2,-1.0)  {};
            \node[vtx] (X3) at (2.0,-0.95) {};
            \node[vtx] (X4) at (0.6,-0.8)  {};
            \draw[restC] (Hb) to[out=-60, in=10] (X1);
            \draw[newedge] (X1) -- (X2);
            \draw[restC] (X2) -- (X3);
            \draw[newedge] (X3) -- (X4);
            \draw[restC] (X4) to[out=170, in=-120] (Ta);
            \node[lbl, gray!70!black] at (6.0,-0.3) {$C$};
            \draw[oldpath]   (Ha) to[out=30, in=150] (Tb);
            \node[lbl] at (2.6,0.8) {$|P_a| \le \lfloor (k-2)/2 \rfloor$};
            \draw[crossedge] (Tb) -- node[lbl, right] {$p_{e_b}$} (R);
            \draw[treepath]  (R) -- node[lbl, above=1.5pt] {$\le d$} (Q);
            \draw[crossedge] (Q) -- node[lbl, left] {$q_{e_a}^{-1}$} (Ha);
            \draw[newedge] (Ta) -- node[lbl, above left] {$e_a$} (Ha);
            \draw[newedge] (Tb) -- node[lbl, above right] {$e_b$} (Hb);
          \end{tikzpicture}}
          \caption{The two cases in the proof of Claim~\ref{claim_3}.
                   New edges of $C$ are solid, old segments of $C$ are dotted, crossing edges of $G$ are dashed, paths inside a layer are wavy, and the rest of $C$ is grey.
                   (a) Case (i): $m = \tp(e_a) = \tp(e_b)$. The walk between $u'$ and $v'$ through $q_{e_a}^{-1}$, $P_a$ and $p_{e_b}$ has length at most $2d+k+1 \le D_0$.
                   (b) Case (ii): $m = \tp(e_b) = \tau(\tp(e_a))$. Hence $q_{e_a}$ and $p_{e_b}$ both reach the same layer $L_m$, closing the short cycle $W$ (shaded) in $G$.}
          \label{fig:claim3}
        \end{figure}

        \emph{Case (i): all new edges of $C$ have the same top $m \in \{n-1,n\}$} (Figure~\ref{fig:claim3}(a)).
            Choose any new edge $e_a$ in $C$ and let $e_b$ and $P_a$ be as above.
            Then $q_{e_a}$ joins $h(e_a)$ to $L_{\tau(m)}$ and $p_{e_b}$ joins $t(e_b)$ to $L_m$.
            Concatenating a path inside $L_{\tau(m)}$ from the vertex of $\{u',v'\}$ in that layer to the endpoint of $q_{e_a}$, the edge $q_{e_a}^{-1}$, the walk $P_a$, the edge $p_{e_b}$,
            and a path inside $L_m$ to the other vertex of $\{u',v'\}$ yields a walk in $G$ connecting $u',v'$, of length at most
            \[ d + 1 + (k-1) + 1 + d \;\le\; 2d + g \;=\; D_0, \]
            contradicting $\mathrm{dist}_G(u',v') > D_0$.
            This includes the cases where $C$ has a single new edge, where consecutive new edges are adjacent ($P_a$ empty), and where new edges arise from half-loops or whole-loops of $H$.

        \emph{Case (ii): $C$ has new edges with both tops} (Figure~\ref{fig:claim3}(b)).
            Going around $C$, the top changes at least twice.
            Hence there are at least two new edges $e_a$ such that the next new edge $e_b$ has $\tp(e_b) = \tau(\tp(e_a))$.
            The corresponding segments $P_a$ are disjoint and $C$ has at least two new edges, so their total length is at most $k-2$.
            Choose such $e_a$ with $|P_a| \le \lfloor (k-2)/2 \rfloor$, and let $m = \tp(e_b) = \tau(\tp(e_a))$.
            Then both $q_{e_a}$ and $p_{e_b}$ join $L_{[n-2]}$ to the same layer $L_m$.
            Let $W$ be the closed walk in $G$ that starts at $h(e_a)$, follows $P_a$ to $t(e_b)$, takes the edge $p_{e_b}$ into $L_m$, follows the path in the copy of $T$ in $L_m$ to the endpoint of $q_{e_a}$, and returns to $h(e_a)$ by $q_{e_a}^{-1}$.
            Its length is at least two and at most
            \[
                \left\lfloor \frac{k-2}{2} \right\rfloor + 2 + d \;\leq\; \frac{g-3}{2} + 2 + d \;=\; \frac{g+1}{2} + d \;<\; g,
            \]
            where the last inequality holds since $g \geq 2d+2$.
            
        It remains to check that $W$ is non-backtracking.
        The walks $P_a$ and the tree path are non-backtracking, and consecutive edges of $W$ of different kinds (low, crossing, high) are not inverses of each other.
        So $W$ can backtrack only if the tree path is empty and $q_{e_a}^{-1} = p_{e_b}^{-1}$, which by~\eqref{eqn:witness_inverse} means $e_b = e_a^{-1}$.
        In that case $P_a$ is non-empty, since $C$ is cyclically non-backtracking, and $P_a$ is a closed non-backtracking walk in $G''$ from $h(e_a)$ to $t(e_a^{-1}) = h(e_a)$, shorter than $C$, contradicting the minimality of $C$.
        Hence $W$ is a closed non-backtracking walk in $G$ of length less than $g$, a contradiction.
    \end{claimproof}
    
    By Claim~\ref{claim_3}, $G''$ is a lift of $H$ with girth at least $g$ and fewer vertices than $G$. So is each of its connected components, contradicting the minimality of $G$.
    Therefore $\diam(G) \leq D_0 = g + 2\diam(T)$, as claimed.
\end{proof}

\begin{lemma}\label{lemma:liftExists}
    For any undirected connected graph $H$ with minimal degree at least two and any $g \geq 1$, there exists a connected covering graph $G$ with $\girth(G) \geq g$.
\end{lemma}

\begin{proof}
    Assume first that $H$ does not have half-loops.
    The proof is by induction on $g$, where the base $g=1$ trivially holds for $G=H$.
    By induction, assume that $G$ is a connected graph covering $H$ with $\girth(G) \geq g-1$, where the number of length $g-1$ cycles $\phi(G)$ is as small as possible.
    Here a cycle of length $g-1$ is a closed non-backtracking walk of length $g-1$.
    As $\girth(G) \geq g-1$, such a walk visits $g-1$ distinct vertices, since otherwise it contains a shorter closed non-backtracking walk.
    Hence it traverses each edge at most once.
    If $\phi(G) = 0$, then $\girth(G) \geq g$, and we are done. 
    Otherwise, assume $\phi(G)>0$ and let $G_2$ be a random 2-lift of $G$. 
    Namely, a random graph with vertex set $V(G) \times \{0,1\}$,
    where for each edge $e \in \dirE(G)$ we pick the permutation $\perm_\pi(e) \in S_2$ independently at random such that $\perm_\pi(e)=\perm_\pi(e^{-1})$.
    We observe that all 2-lifts of $G$ have girth at least $g-1$, and argue that among them there is a 2-lift with a smaller number of length $g-1$ cycles. 
    Indeed, as every length $g-1$ cycle $C$ in $G$ traverses each edge at most once, the product of $\perm_\pi(e)$ for $e \in C$ is uniform in $S_2$.
    Therefore $C$ is turned into either one cycle of length $2(g-1)\geq g$ or two cycles of length $g-1$, with equal probabilities,
    and it follows that the expected value of $\phi(G_2)$ is $\phi(G)$.
    However, since the lift with $\perm_\pi(e)=\text{id}$ for all $e$ has $2\phi(G)$ length $g-1$ cycles,
    it follows that there exists some 2-lift $G'$ of $G$ with $\phi(G') < \phi(G)$.
    If $G'$ were disconnected then, as each of its connected components is a lift of $G$, it would consist of two copies of $G$ and have $2\phi(G)$ cycles of length $g-1$.
    Hence $G'$ is a connected lift of $H$ with $\girth(G') \geq g-1$ and $\phi(G') < \phi(G)$, contradicting the choice of $G$.

    If $H$ has half-loops, let $H'$ be the 2-lift of $H$ in which the permutation of every half-loop is the transposition and all other permutations are the identity.
    Then the lift of a half-loop $e$ with $h(e)=t(e)=v$ is a single edge connecting the two vertices covering $v$, and $H'$ is connected, as each of its two layers induces $H$ with its half-loops removed, which is connected, and these edges join the two layers.
    As the graph $H'$ covers $H$ and does not have any half-loops, the required result is obtained by applying the preceding proof to $H'$.
\end{proof}

\section{Three Generalizations for Non-regular Graphs}\label{section:comparing}
Returning to our motivating problem of finding the right generalization of $n(k,g)$ to non-regular graphs, we propose three possible definitions.
Let $n({\cal G},g)$ denote the order of the smallest graph of girth at least $g$ in the graph family ${\cal G}$ and consider the following graph families for some fixed graph $H$:
\begin{eqnarray*}
    {\cal L}(H) & = & \{ G : G \mbox{ is a connected lift of } H \}, \\
    {\cal U}(H) & = & \{ G : G \mbox{ is connected and }\tilde{G} = \tilde{H} \}, \\
    {\cal D}(H) & = & \{ G : G \mbox{ is connected and has the same proportion of each degree as } H \}.
\end{eqnarray*}
Respectively, $n(H,g)=n({\cal L}(H),g)$ is the smallest order of a lift of $H$ of girth at least $g$,
$n({\cal U}(H),g)$ is the smallest order of a graph of girth at least $g$ with the same universal cover as $H$,
and $n({\cal D}(H),g)$ is the smallest order of a graph of girth at least $g$ with the same proportion of each degree as $H$.
To compare the three families, we use the following result of Leighton.

\begin{proposition}[Leighton~\cite{leighton1982finite}, Theorem and Section~4]\label{proposition:leighton}
    Two finite connected graphs $G_1, G_2$ have a common finite cover if and only if they have isomorphic universal covers.
    Moreover, in that case they have a connected common cover with at most $c \cdot |V(G_1)| \cdot |V(G_2)|$ vertices, where $c$ depends only on the universal cover.
\end{proposition}

\noindent
Leighton allows loops, and notes that his results hold whether a loop contributes one or two to the degree of its vertex, namely, for both our half-loops and whole-loops.
The cover constructed in~\cite{leighton1982finite} may be disconnected, but each of its connected components is a common cover of $G_1$ and $G_2$ with at most as many vertices.

Clearly ${\cal L}(H) \subseteq {\cal U}(H)$: if $G$ covers $H$, then $\tilde{G}$ covers $H$ as a composition of covers, and since $\tilde{G}$ is a tree, it is the universal cover of $H$; see Facts~1 and~2 in~\cite{leighton1982finite}. 
Also ${\cal U}(H) \subseteq {\cal D}(H)$: for $G \in {\cal U}(H)$, let $K$ be a connected common cover of $G$ and $H$, of height $a$ over $H$ and $b$ over $G$.
For every $d$, the number of degree $d$ vertices of $K$ is $a$ times that number in $H$ and $b$ times that number in $G$, while $|V(K)| = a|V(H)| = b|V(G)|$.
Hence $G$ and $H$ have the same proportion of vertices of each degree.
It follows that for every $g$:
\begin{equation}\label{eqn:naive_graph_size}
    n({\cal L}(H),g) \geq n({\cal U}(H),g) \geq n({\cal D}(H),g).
\end{equation}

\noindent
It turns out that the first inequality cannot be too far from equality:
\begin{lemma}\label{lemma:LU}
For every finite, undirected and connected graph $H$, there exists a constant $c(H)$ such that for every $g$:
\begin{equation*}
c(H) \cdot n({\cal U}(H),g) \geq n({\cal L}(H),g) \geq n({\cal U}(H),g).
\end{equation*}
\end{lemma}

\begin{proof}
The right inequality is part of~\eqref{eqn:naive_graph_size}.
For the left inequality, let $G \in {\cal U}(H)$ with $\girth(G) \geq g$ and $|V(G)| = n({\cal U}(H),g)$.
By Proposition~\ref{proposition:leighton}, $G$ and $H$ have a connected common cover $K$ with $|V(K)| \leq c \cdot |V(H)| \cdot |V(G)|$, where $c$ depends only on $\tilde{H}$.
Since $K$ covers $G$, the cover map takes closed non-backtracking walks in $K$ to closed non-backtracking walks in $G$ of the same length, so $\girth(K) \geq \girth(G) \geq g$.
As $K$ is a lift of $H$, it follows that $n({\cal L}(H),g) \leq |V(K)| \leq c(H)\, n({\cal U}(H),g)$ with $c(H) = c\,|V(H)|$.
\end{proof}

\noindent
For the degree distribution based lower bound, we have the following result:
\begin{proposition}[Alon, Hoory, Linial,~\cite{alon2002moore}]\label{proposition:AHL2002}
Let $G=(V,E)$ be an undirected girth $g$ graph on $n$ vertices with minimal degree at least two, and let:
$$\Lambda = \prod_{v \in V} {(d_v-1)}^{d_v/|\dirE|}, \;\;\;\;\;
n_0(x+1,g) = \sum_{i=0}^{\lfloor(g-1)/2\rfloor} x^i + \sum_{i=0}^{\lfloor g/2-1\rfloor} x^i.$$
Then:
\begin{enumerate}
    \item $n \geq n_0(\Lambda+1,g)$.
    \item $\Lambda \geq \overline{d} - 1$, where $\overline{d}$ is the average degree of $G$, with equality only if $G$ is regular.
    \item 
    $\frac{1}{|\dirE|} \one^T (B_G)^r \one \geq \Lambda^r$ for all $r \geq 0$.
\end{enumerate}
\end{proposition}

\noindent
It follows that:
\begin{equation}\label{eqn:dist-lb}
    n({\cal D}(H),g) \geq n_0(\Lambda(H)+1,g) \geq n_0(\overline{d}(H),g).
\end{equation}

\noindent
This result is related to the $\rho$ based lower bound of Theorem~\ref{theorem:main} by the following proposition by~\cite{eisner2024entropy}, where the authors prove that $\rho(H) \geq \Lambda(H)$ and specify equivalent conditions for equality to hold.
\begin{proposition}[Eisner, Hoory,~\cite{eisner2024entropy}]\label{proposition:rhoLambda}
    For a finite undirected and connected graph $H$ with minimal degree at least two and maximal degree greater than two, we have $\rho(H) \geq \Lambda(H)$.
    Furthermore, the following three conditions are equivalent:
    \begin{enumerate}
        \item $\rho(H) = \Lambda(H)$
        \item For every non-backtracking cycle $C$ the geometric average of vertex degrees minus one along the cycle is equal to $\Lambda$,
            \begin{equation*}
                \left(\prod_{v \in C} (\deg(v)-1)\right)^{\frac{1}{|C|}} = \Lambda(H).
            \end{equation*}
        \item For every non-backtracking path $P$, where all its internal vertices (if any) have degree two and its endpoints have degree larger than two, 
            \begin{equation*}
                \left(\prod_{v \in P} (\deg(v)-1)\right)^{\frac{1}{2|P|}} = \Lambda(H).
            \end{equation*}
        \end{enumerate}
\end{proposition}

\noindent
By combining Propositions~\ref{proposition:AHL2002} and~\ref{proposition:rhoLambda} with inequalities~(\ref{eqn:dist-lb}), (\ref{eqn:naive_graph_size}), Theorem~\ref{theorem:treeBallSize} and Lemma~\ref{lemma:mooreLift}, we obtain the following inequalities, which hold for every $g$ with suitable positive constants $a_1=a_1(H)$ and $a_2=a_2(H)$:
\begin{equation}\label{eqn:complete-diagram}
    \begin{matrix}
        n(H,g)                 & \geq & n({\cal D}(H),g) &      \\
        \vgeq                  &      & \vgeq            &      \\
        a_1 \rho(H)^{g/2}      & \geq & a_2 \Lambda(H)^{g/2} &  \geq a_2 (\overline{d}(H)-1)^{g/2}.
    \end{matrix}
\end{equation}
For instance, one can take $a_1 = c_1/\rho(H)$, where $c_1$ is the constant of Theorem~\ref{theorem:treeBallSize}, and $a_2 = \min\{a_1, 1/\Lambda(H)\}$.

\noindent
At this point, it is worthwhile to compare the values of $\rho$, $\Lambda$ and $\overline{d}$ for some specific graphs.
Consider the three base graphs depicted in Figure~\ref{fig:three_graphs}:
\begin{figure}[h]
  \centering
  \includegraphics[width=0.7\textwidth,clip,trim=0mm 55mm 0mm 55mm]{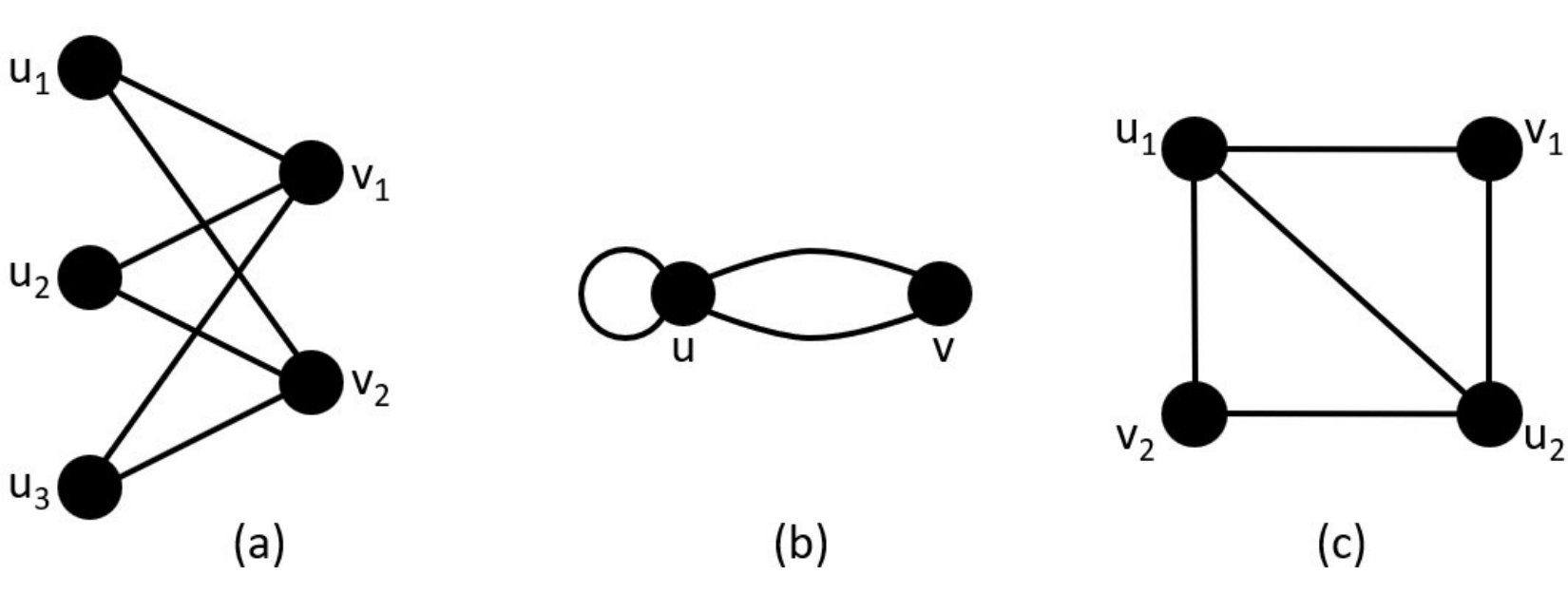}
  \caption{(a) The bipartite $K_{3,2}$ graph, (b) $H_{23}$ with a degree 2 vertex $v$ and a degree 3 vertex $u$ with a half-loop edge, and (c) the diamond graph ($K_4$ minus an edge), which is a 2-lift of $H_{23}$.}
  \label{fig:three_graphs}
\end{figure}

\begin{description}
\item[a. The complete bipartite graph $K_{3,2}$]
\begin{itemize}
    \item[]
    \item $\overline{d}(K_{3,2})-1=(3+3+2+2+2)/5 - 1 = 1.4$
    \item $\Lambda(K_{3,2}) = 2^{(2 \cdot 3)/12}=\sqrt{2}$.
    \item The universal cover of $K_{3,2}$ is the 3-regular infinite tree with a degree 2 vertex added to the middle of each edge. Therefore, $B_r(\cdot)$ has alternating layers of degree 2 and degree 3 vertices, as seen in Figure~\ref{fig:infinite_trees}~(a), yielding $\rho(K_{3,2}) = \sqrt{2}$. Note that $\rho$ can also be computed as the Perron eigenvalue of a 12 by 12 matrix, $B_{K_{3,2}}$ that can be reduced to a 2 by 2 matrix, as there are only two directed edge types $(v,u),(u,v)$:
    \begin{equation*}
        \rho(K_{3,2}) = \rho\left( \left[
            \begin{tabular}{cc} 0&1\\ 2&0 \end{tabular}
        \right] \right) = \sqrt{2}.
    \end{equation*}
\end{itemize}
\item[b. The graph $H_{23}$]
\begin{itemize}
    \item[]
    \item $\overline{d}(H_{23})-1=(3+2)/2 - 1 = 1.5$
    \item $\Lambda(H_{23}) = 2^{3/5}=1.5157\ldots$.
    \item 
    As $H_{23}$ has 5 directed edges, $\rho(H_{23})$ is the Perron eigenvalue of the 5 by 5 matrix $B_{H_{23}}$. But, as there are only 3 directed edge types: $(u,v)$, $(u,u)$, $(v,u)$, it suffices to compute:
    \begin{equation*}
        \rho(H_{23}) = \rho(B_{H_{23}}) = \rho\left( \left[
            \begin{tabular}{ccc} 0&0&1\\ 2&0&0\\ 1&1&0 \end{tabular}
        \right] \right) = 1.5214\ldots
    \end{equation*}
\end{itemize}
\item[c. The diamond graph ($K_4$ minus an edge)]
\begin{itemize}
    \item[]
    \item
    The graph is a 2-lift of $H_{23}$, so $\overline{d}$, $\Lambda$ and $\rho$ are the same.
\end{itemize}
\end{description}

It should be noted that because all three graphs are non-regular, we have strict inequality $\Lambda(H) > \overline{d}(H)-1$, in agreement with Proposition~\ref{proposition:AHL2002}. Furthermore, $\rho(H) \ge \Lambda(H)$ with equality for $H=K_{3,2}$ and strict inequality for $H=H_{23}$, in agreement with Proposition~\ref{proposition:rhoLambda}: 
\begin{eqnarray*}
\rho(K_{3,2}) &=& \Lambda(K_{3,2}) = \sqrt{2},\\
\rho(H_{23}) = 1.5214\ldots &>& \Lambda(H_{23}) = 2^{0.6} = 1.5157\ldots.
\end{eqnarray*}

\begin{figure}[h]
  \centering
  \subfloat[]{{\includegraphics[width=0.50\textwidth,clip,trim=40mm 10mm 40mm 10mm]{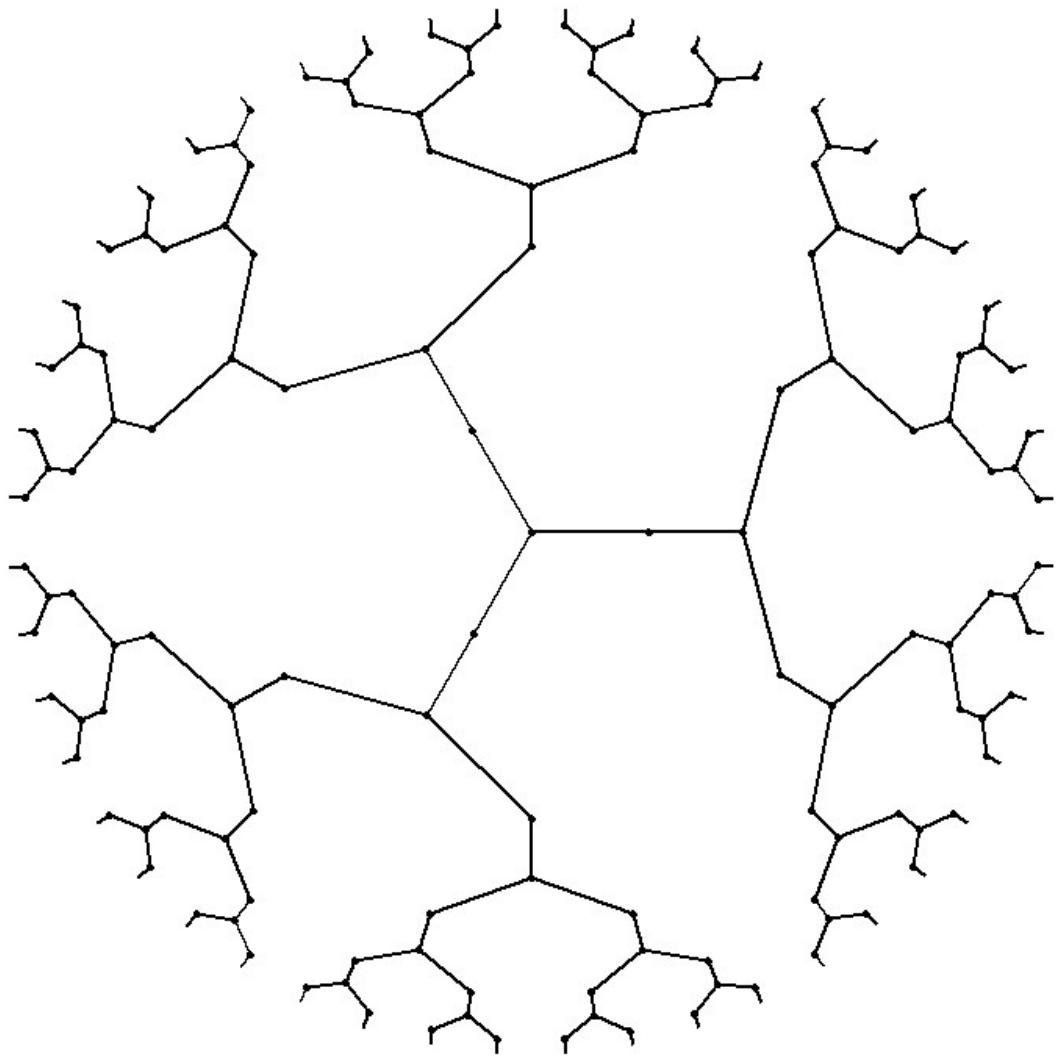}}} 
  \qquad
  \subfloat[]{{\includegraphics[width=0.45\textwidth,clip,trim=35mm 0mm 35mm 0mm]{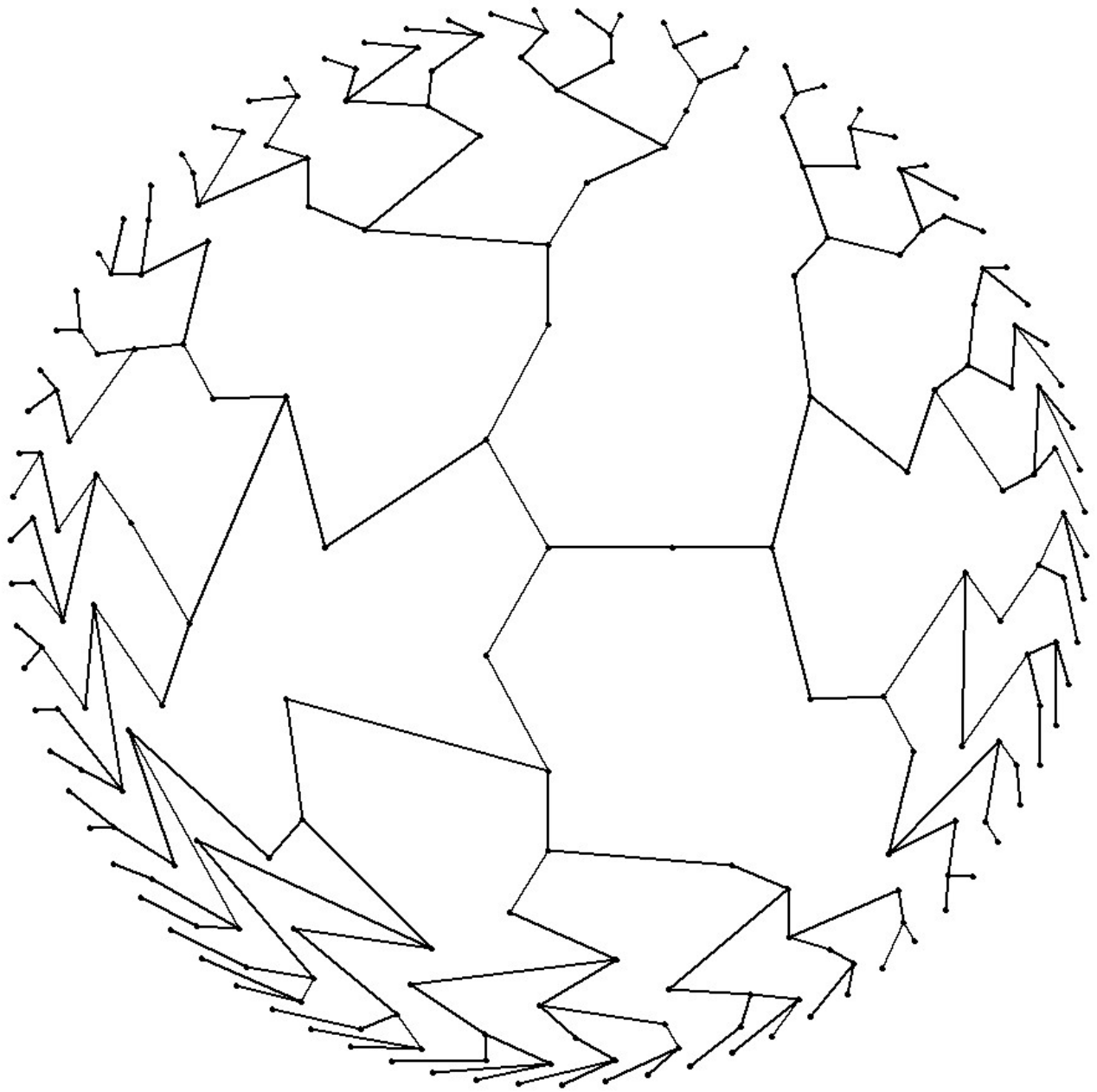}}}
  \caption{Two infinite trees: (a) $\widetilde{K_{3,2}}$ with $\rho = \sqrt{2}$, (b) $\widetilde{H_{23}}$ with $\rho = 1.5214\ldots$.}
  \label{fig:infinite_trees}
\end{figure}

\vskip 5pt \noindent The following diagram summarizes what we know about $H_{23}$, where $a_1,a_2$ are the constants from \eqref{eqn:complete-diagram} for $H=H_{23}$:
\begin{equation}\label{eqn:h23-diagram}
    \begin{matrix}
        n(H_{23},g)                & \geq & n({\cal D}(H_{23}),g)     &     \\
        \vgeq                      &      & \vgeq                     &          \\
        a_1\,(1.5214\ldots)^{g/2}  & >    & a_2\,(1.5157\ldots)^{g/2}.
    \end{matrix}
\end{equation}

\section{Concrete Constructions and Bounds for the $H_{23}$ Graph}\label{section:concrete}

In this section we explore non-asymptotic upper and lower bounds on the value of $n(H,g)$, for the graph $H=H_{23}$ depicted in Figure~\ref{fig:three_graphs}~(b). 
We discuss this in the following subsections: The Moore lower bound, The Erd\H{o}s-Sachs Upper Bound, Construction Algorithms, Exhaustive Search, Results and Data Availability.

\subsection{The Moore Lower bound}
The Moore lower bound for lifts is stated in Lemma~\ref{lemma:mooreLift}, where the size of the ball centered around some tree vertex can be obtained by \eqref{eqn:tree_ball_size} and \eqref{eqn:tree_ball_size2}, 
with $g = 2r + 1$:
\begin{equation}\label{eqn:exac_Moore_odd}
    n(H,g) 
    \geq n_0(H,g) 
    = \max_{v \in V(\tilde{H})} |B_r(v)| 
    = 1 + \max_{v \in V(H)} \sum_{i=1}^r \delta_v^T \left(B_H\right)^{i-1} \one,
\end{equation}
where $\delta_v$ is the indicator function of $t^{-1}(v)$. 
Similarly, for $g=2r$:
\begin{equation}\label{eqn:exac_Moore_even}
    n(H,g) 
    \ge n_0(H,g)
    = \max_{e \in E(\tilde{H})} |B_{r-1}(e)|
    = \max_{e \in \dirE(H)} \sum_{i=1}^{r} (\delta_e+\delta_{e^{-1}})^T \left(B_H\right)^{i-1} \one,
\end{equation}
where $\delta_e$ is the indicator function of $e$.

\subsection{The Erd\H{o}s-Sachs Upper Bound}
We compute the Erd\H{o}s-Sachs upper bound for $n(H,g)$ using Corollary~\ref{corollary:ESlift}.
As the spanning tree $T$ has diameter one, it follows that for all $g \geq g_0 = 4$:
\begin{equation*}
  n(H,g) \leq \min_{v \in V(\tilde{H})} |B_{g+2}(v)|
    = 1 + \min_{v \in V(H)} \sum_{i=1}^{g+2} \delta_v^T \left(B_H\right)^{i-1} \one.
\end{equation*}

\subsection{Greedy Construction Algorithms}
In order to get a clearer picture regarding the behavior of $n(H,g)$ we made several attempts to construct small girth $g$ graphs that cover $H_{23}$ by randomized greedy algorithms. Our starting point, Algorithm A, is an adaptation to $H_{23}$ of the Linial and Simkin~\cite{linial2021randomized} algorithm for regular graphs. 
The algorithms we describe are split into two categories:
\begin{enumerate}
    \item 
        Start with the initial graph $G_0$ being the cycle $C_n$ with vertex set $\{0,\ldots n-1\}$ where $i$ is connected to $i \pm 1 \pmod n$. Keep adding edges, one at a time, until either the required graph is obtained, or there is no suitable edge to add. 
        In step $k$, given the current graph $G_k$, let $P_k$ denote the graph of permissible new edges, where $V(P_k)$ is the set of deficient $G_k$ vertices, $V(P_k)=\{i : i \mbox{ is even and }\deg_{G_k}(i)=2\}$, and $E(P_k)$ is the set of edges that can be added to $G_k$ without creating a cycle shorter than $g$. 
        In step $k$ the algorithm randomly picks an edge $e$ in $E(P_k)$ and adds it to $G_k$, where algorithms A,B and C differ only in the way they pick the edge $e$.  
        For each $n,g$ pair, we perform multiple runs of the selected algorithm and check if at least one is successful. 
    \item
        Start with the initial graph $G_0$ being either the diamond graph, depicted in Figure~\ref{fig:three_graphs}~(c), or the minimal order girth $g-1$ graph obtained so far. 
        If at step $k$ we have $\girth(G_k) \geq g$ then stop; Otherwise, randomly choose two edges $e,f \in E(G_k)$, where the endpoints of each edge have different degrees, and obtain the graph $G_{k+1}$ by applying the operation shown in Figure~\ref{fig:growth_transform_k4m}.
        Algorithms GD, GF and GF2 fall into this framework, differing only in the way they pick $e,f$. 
        For each value of $g$ we perform multiple runs of the selected algorithm and take the resulting graph of minimal order. 
        \begin{figure}[h]
            \centering
            \includegraphics[width=0.7\textwidth,clip,trim=40mm 75mm 40mm 75mm]{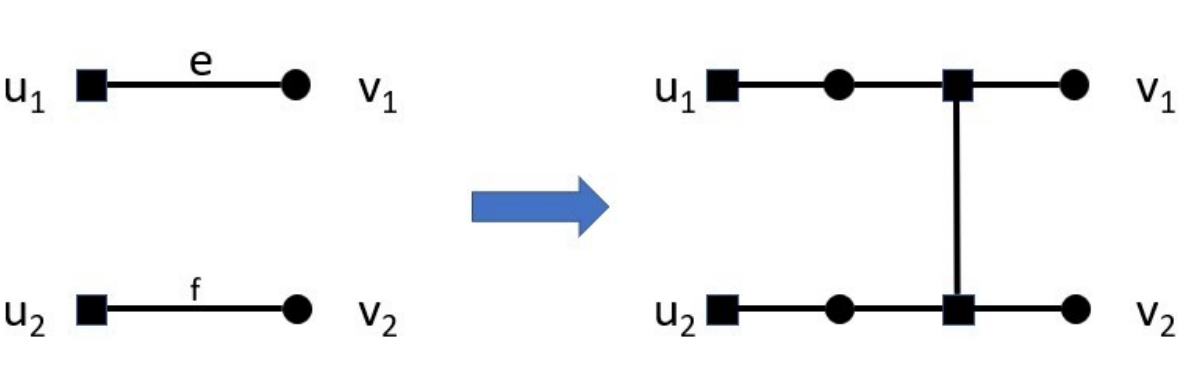}
            \caption{Transform a covering graph of $H$ by picking two edges $e=(u_1,v_1)$, $f=(u_2,v_2)$ with $\deg(u_i)=3$ (marked by a square) and $\deg(v_i)=2$ (marked by a circle). The new graph has 4 new vertices, 2 of degree two and 2 of degree three, and also covers $H$.}
            \label{fig:growth_transform_k4m}
        \end{figure}
\end{enumerate}

\begin{description}
\item{Algorithm A} - Start with the cycle $C_n$ as above, where in step $k$ we pick $e \in E(P_k)$ uniformly at random.
\item{Algorithm B} - The same, but in step $k$ we pick a uniformly random $u \in V(P_k)$ and then pick $e=(u,v) \in E(P_k)$ where $v$ is a uniformly random $P_k$ neighbor of $u$.
\item{Algorithm C} - The same, but in step $k$ we pick a uniformly random $u \in V(P_k)$ from the vertices with smallest $\deg_{P_k}(u)$ and then pick $e=(u,v) \in E(P_k)$ where $v$ is a uniformly random $P_k$ neighbor of $u$.
\item{Algorithm GD} - In each step, one picks a uniformly random edge $e$ on a cycle of length smaller than $g$, and then chooses a random $f$ from the edges that are most distant from $e$. 

\item{Algorithm GF} - 
    In each step, one picks a uniformly random edge $e$ maximizing the vector $c(e)=(c_1(e),c_2(e),\ldots,c_{g-1}(e))$ in lexicographic order, 
    where $c_l(e)$ is the number of non-backtracking cycles of length $l$ in which $e$ participates. 
    Given $e=(u_1,v_1)$ and $f=(u_2,v_2)$, define $D(e,f)=\dist(v_1,v_2)+3$.
    The edge $f$ is chosen uniformly at random among the edges maximizing $\bigl(\min\{D(e,f),g\},\, c(f)\bigr)$ in lexicographic order.

\item{Algorithm GF2} - A variant of GF that considers several candidates for $e$ in each step.
    Let ${\cal E}$ be the set of edges maximizing $c(e)$ in lexicographic order, from which GF picks $e$.
    GF2 picks a uniformly random subset of $\min\{s,|{\cal E}|\}$ edges of ${\cal E}$, where $s=16$, computes a partner $f$ for each of them,
    and applies the operation of Figure~\ref{fig:growth_transform_k4m} to a uniformly random pair $(e,f)$ maximizing $\min\{D(e,f),g\}$.
    In GF2a the partner $f$ of $e$ is a uniformly random edge maximizing $D(e,f)$, and in GF2b it is chosen as in GF, so that GF2b with $s=1$ is GF.
\end{description}

As shall be seen in Subsection~\ref{subsection:results}, the later algorithms generally do better: for $g \geq 21$ each algorithm in the list improves on the previous one, while for smaller values of $g$ there are some exceptions.

\subsection{Exhaustive Search}\label{subsection:exhaustive}

For small values of $g$ and $n$ it is possible to perform an exhaustive search on the lifts of $H_{23}$, and either find one of girth at least $g$ and order at most $n$, proving that $n(H_{23},g) \le n$, or show that no such lift exists, proving that $n(H_{23},g) > n$.
The procedure is based on the following description of the lifts of $H_{23}$.
Recall that $H_{23}$ has a degree three vertex $u$ carrying a half-loop, and a degree two vertex $v$ joined to $u$ by two parallel edges.

\begin{lemma}\label{lemma:cycletype}
    Let $G$ be a lift of $H_{23}$ on $n$ vertices without half-loops, with $U$ and $V$ being the fibers of $u$ and $v$.
    Then $n$ is divisible by $4$, the edges of $G$ with both endpoints in $U$ form a perfect matching $M$ on $U$, and $G - M$ is a disjoint union of cycles whose vertices alternate between $U$ and $V$. 
    In particular, their cycle lengths are even numbers $2\ell_1,\ldots,2\ell_k$ that sum up to $n$.
    
    Conversely, let $n$ be a natural number divisible by 4, let $C$ be a disjoint union of cycles of lengths $2\ell_1,\ldots,2\ell_k$ of total sum $n$, 
    let the set $U$ contain every second vertex of each cycle in $C$ and let $M$ be a perfect matching on $U$.
    Then $C+M$ is a lift of $H_{23}$.
\end{lemma}

\begin{proof}
    Since the cover map is a bijection between the edges emanating from each vertex and those emanating from its image, the vertices of $U$ have degree three and those of $V$ have degree two, and $|U| = |V| = n/2$, as all fibers have the same size.
    The two edges at a vertex of $V$ cover the two edges joining $v$ to $u$, so both neighbors of a vertex of $V$ lie in $U$.
    At a vertex $x \in U$, the edge covering the half-loop has its other endpoint in $U$, and it is not a half-loop of $G$ by assumption; the other two edges at $x$ lead to $V$.
    Hence the edges inside $U$ form a perfect matching $M$ on $U$, so $n/2=|U|$ is even, and every vertex of $G-M$ has degree two, with both of its neighbors on the other side.
    Thus $G-M$ is a disjoint union of cycles alternating between $U$ and $V$, and in particular of even lengths, where a cycle of length two consists of two parallel edges.
    
    For the converse, the edges of $C$ form a disjoint union of even cycles, so they can be colored by two colors $1,2$ such that each vertex is incident to one edge of each color.
    Mapping $U$ to $u$, $V$ to $v$, the edges of $M$ to the half-loop, and the edges of color $i$ to the $i$-th edge joining $u$ and $v$ is a covering map.
\end{proof}

\noindent
Hence, for $g \ge 3$, a lift of $H_{23}$ on $n$ vertices with girth at least $g$ exists if and only if there is a partition $\lambda=(\ell_1,\ldots,\ell_k)$ of $n/2$ into parts of size at least $\lceil g/2 \rceil$, called the \emph{cycle type}, and a perfect matching $M$ of $U$ such that $C_\lambda + M$ has girth at least $g$, where $C_\lambda$ is the disjoint union of cycles of lengths $2\ell_1,\ldots,2\ell_k$.
For each cycle type, Algorithm~\ref{alg:exhaustive} searches for $M$.
The vertices of $C_\lambda$ are ordered cycle by cycle, each cycle starting at a vertex of $U$, its \emph{leader}, and a cycle is \emph{untouched} if none of its vertices is matched.
The algorithm matches the first unmatched vertex $x \in U$ to every partner $y$ with $\dist(x,y) \geq g-1$, so that no cycle shorter than $g$ is closed, subject to three symmetry reductions:
\begin{enumerate}
    \item[(R1)] An untouched cycle other than the cycle of $x$ is entered only at its leader.
    \item[(R2)] Among untouched cycles of the same length, other than the cycle of $x$, only the first one is entered.
    \item[(R3)] If the cycle of $x$ is untouched, then a partner $y$ on that cycle is tried only if it is at most half way around the cycle from $x$, in the direction of the vertex ordering.
\end{enumerate}
Each reduction is justified by an automorphism of the current graph that fixes $x$ and all matched vertices: a rotation of an untouched cycle for (R1), an exchange of two untouched cycles of equal length for (R2), and the reflection of the untouched cycle of $x$ that fixes $x$ for (R3).

\begin{algorithm}[ht]
\caption{The search for a matching $M$ of a given cycle type $\lambda$. Initially $G = C_\lambda$.}
\label{alg:exhaustive}
\begin{algorithmic}[1]
\Procedure{Search}{$G$}
    \If{every vertex of $U$ is matched}
        \State \Return \textsc{true}
    \EndIf
    \State $x \gets$ the first unmatched vertex of $U$, and let $C_x$ be its cycle
    \ForAll{unmatched $y \in U$ after $x$ with $\dist_G(x,y) \geq g-1$ (possibly infinite)}
        \If{$y$ lies on an untouched cycle $D \neq C_x$}
            \If{$y$ is not the leader of $D$, or an untouched $D' \neq C_x$ before $D$ has $|D'|=|D|$}
                \State \textbf{continue} \Comment{(R1), (R2)}
            \EndIf
        \ElsIf{$y \in C_x$, $C_x$ is untouched, and $y$ is more than $|C_x|/2$ steps after $x$ on $C_x$}
            \State \textbf{continue} \Comment{(R3)}
        \EndIf
        \If{\Call{Search}{$G + xy$}}
            \State \Return \textsc{true}
        \EndIf
    \EndFor
    \State \Return \textsc{false}
\EndProcedure
\end{algorithmic}
\end{algorithm}

\noindent
To show that $n(H_{23},g) > n$, the search must fail for every cycle type of $n/2$.
For $g \leq 15$, this was done for every multiple of $4$ from the Moore bound up to the lower bound of Table~\ref{table:results} minus $4$, with up to $205$ cycle types for a single value of $n$.
As a check of the reductions, every run of the exhaustive search with $n \le 64$ was repeated with (R1)--(R3) turned off, and the verdicts agreed for every cycle type.
The running time grows by roughly an order of magnitude for every additional $4$ vertices, from seconds for $n \leq 68$ and about an hour for $n=80$ to more than half a day of CPU time for each of $n=84$ and $n=96$, and about two weeks for $n=100$.
The source code and its output, listing the verdict for each cycle type, are available as described in Section~\ref{subsection:data}.

\subsection{Result Summary}\label{subsection:results}
By Lemma~\ref{lemma:cycletype}, the order of a lift of $H_{23}$ of girth at least $3$ is divisible by 4, so lower bounds can be rounded up and upper bounds can be rounded down to a multiple of 4, as done in Table~\ref{table:results}.
For each $g$, the table lists the Moore bound $n_0(H,g)$ of~\eqref{eqn:exac_Moore_odd} and~\eqref{eqn:exac_Moore_even}, the best lower and upper bounds, the smallest order of a graph found by each construction algorithm, and the Erd\H{o}s-Sachs bound $n_{ES}(H,g)$.
The best lower bound is the Moore bound, improved by the exhaustive search for $g \in \{7,9,11,12,13,14,15\}$, and the best upper bound is the smallest order found by any of the construction algorithms or by the exhaustive search.
For all $g \leq 14$ the two bounds coincide, so that $n(H_{23},g)$ is determined exactly.
For $g \in \{8,10,12,13,14,15\}$ the exhaustive search found smaller graphs than the construction algorithms, while for $g \geq 16$ the best upper bound is due to GF2: to GF2a for $g \in \{16,17\}$ and to GF2b for $g \geq 18$.
The results are shown in Table~\ref{table:results} and Figure~\ref{fig:graph}, and Figure~\ref{fig:cages} depicts, for each $g \leq 10$, a lift of $H_{23}$ of girth at least $g$ and order $n(H_{23},g)$, namely an $(H_{23},g)$-cage.

\begin{figure}[!htbp]
    \centering
    \includegraphics[width=\textwidth]{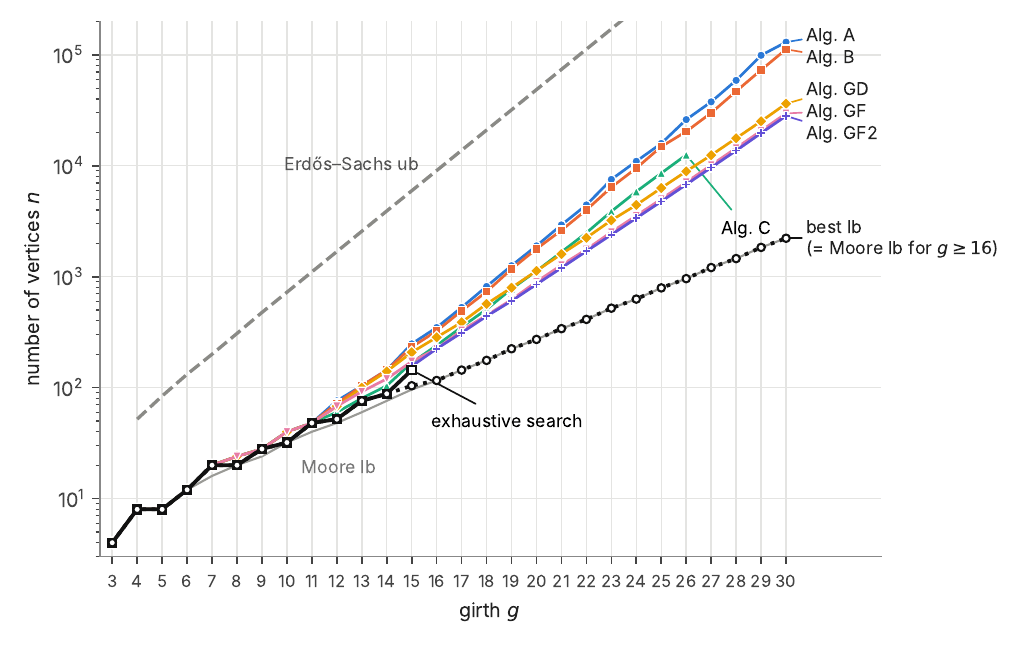}
    \caption{Lower and upper bounds on $n(H_{23},g)$, and the smallest graphs found by each construction algorithm. For $g \leq 15$ the best upper bound is due to the exhaustive search, and for GF2 the better of its two variants is shown.}
    \label{fig:graph}
\end{figure}

\begin{table}[!htbp]
\centering\footnotesize\setlength{\tabcolsep}{4.5pt}
\begin{tabular}{|r|r|r|r|r|r|r|r|r|r|r|r|}
\hline
    & \multicolumn{2}{|c|}{Lower bounds} & \multicolumn{9}{|c|}{Upper bounds}\\
    \cline{2-12}
$g$ & $n_0(H,g)$ & best & best & A & B & C & GD & GF & GF2a & GF2b & $n_{ES}(H,g)$ \\ \hline
        3 &         4 &         4 &         4 &         4 &         4 &         4 &         4 &         4 &           &           &        -- \\ \hline
        4 &         8 &         8 &         8 &         8 &         8 &         8 &         8 &         8 &           &           &        52 \\ \hline
        5 &         8 &         8 &         8 &         8 &         8 &         8 &         8 &         8 &           &           &        84 \\ \hline
        6 &        12 &        12 &        12 &        12 &        12 &        12 &        12 &        12 &           &           &       132 \\ \hline
        7 &        16 &        20 &        20 &        20 &        20 &        20 &        20 &        20 &           &           &       200 \\ \hline
        8 &        20 &        20 &        20 &        24 &        24 &        24 &        24 &        24 &           &           &       308 \\ \hline
        9 &        24 &        28 &        28 &        28 &        28 &        28 &        28 &        28 &           &           &       476 \\ \hline
       10 &        32 &        32 &        32 &        40 &        40 &        40 &        40 &        40 &           &           &       724 \\ \hline
       11 &        40 &        48 &        48 &        48 &        48 &        48 &        48 &        48 &           &           &     1,104 \\ \hline
       12 &        48 &        52 &        52 &        76 &        72 &        60 &        68 &        68 &           &           &     1,684 \\ \hline
       13 &        60 &        76 &        76 &       104 &       104 &        80 &       100 &        92 &           &           &     2,564 \\ \hline
       14 &        76 &        88 &        88 &       144 &       144 &       104 &       140 &       120 &           &           &     3,908 \\ \hline
       15 &        96 &       104 &       144 &       248 &       232 &       168 &       208 &       172 &       156 &       172 &     5,944 \\ \hline
       16 &       116 &       116 &       224 &       348 &       324 &       240 &       284 &       224 &       224 &       232 &     9,044 \\ \hline
       17 &       144 &       144 &       312 &       528 &       492 &       352 &       388 &       328 &       312 &       316 &    13,772 \\ \hline
       18 &       176 &       176 &       440 &       816 &       732 &       508 &       568 &       448 &       452 &       440 &    20,948 \\ \hline
       19 &       224 &       224 &       604 &     1,256 &     1,172 &       776 &       796 &       620 &       620 &       604 &    31,872 \\ \hline
       20 &       272 &       272 &       852 &     1,896 &     1,780 &     1,124 &     1,128 &       900 &       884 &       852 &    48,500 \\ \hline
       21 &       340 &       340 &     1,200 &     2,928 &     2,620 &     1,668 &     1,600 &     1,272 &     1,248 &     1,200 &    73,780 \\ \hline
       22 &       412 &       412 &     1,696 &     4,432 &     3,996 &     2,484 &     2,240 &     1,760 &     1,740 &     1,696 &   112,260 \\ \hline
       23 &       520 &       520 &     2,388 &     7,524 &     6,400 &     3,872 &     3,220 &     2,524 &     2,460 &     2,388 &   170,792 \\ \hline
       24 &       628 &       628 &     3,368 &    10,964 &     9,504 &     5,848 &     4,428 &     3,548 &     3,488 &     3,368 &   259,828 \\ \hline
       25 &       792 &       792 &     4,760 &    15,880 &    14,860 &     8,568 &     6,296 &     5,008 &     4,916 &     4,760 &   395,324 \\ \hline
       26 &       960 &       960 &     6,764 &    26,048 &    20,340 &    12,524 &     8,900 &     7,100 &     7,008 &     6,764 &   601,428 \\ \hline
       27 &     1,208 &     1,208 &     9,624 &    37,656 &    29,936 &           &    12,448 &    10,124 &     9,992 &     9,624 &   914,992 \\ \hline
       28 &     1,456 &     1,456 &    13,680 &    58,792 &    46,900 &           &    17,692 &    14,428 &    14,280 &    13,680 & 1,392,084 \\ \hline
       29 &     1,836 &     1,836 &    19,660 &    98,872 &    72,940 &           &    25,144 &    20,644 &    20,472 &    19,660 & 2,117,860 \\ \hline
       30 &     2,220 &     2,220 &    28,096 &   129,988 &   112,016 &           &    36,208 &    29,616 &    29,300 &    28,096 & 3,222,084 \\ \hline
\end{tabular}
\caption{Bounds on $n(H_{23},g)$. Empty entries were not computed.}
\label{table:results}
\end{table}

\begin{figure}[!htbp]
    \centering
    \includegraphics[width=\textwidth,clip,trim=0mm 75mm 0mm 75mm]{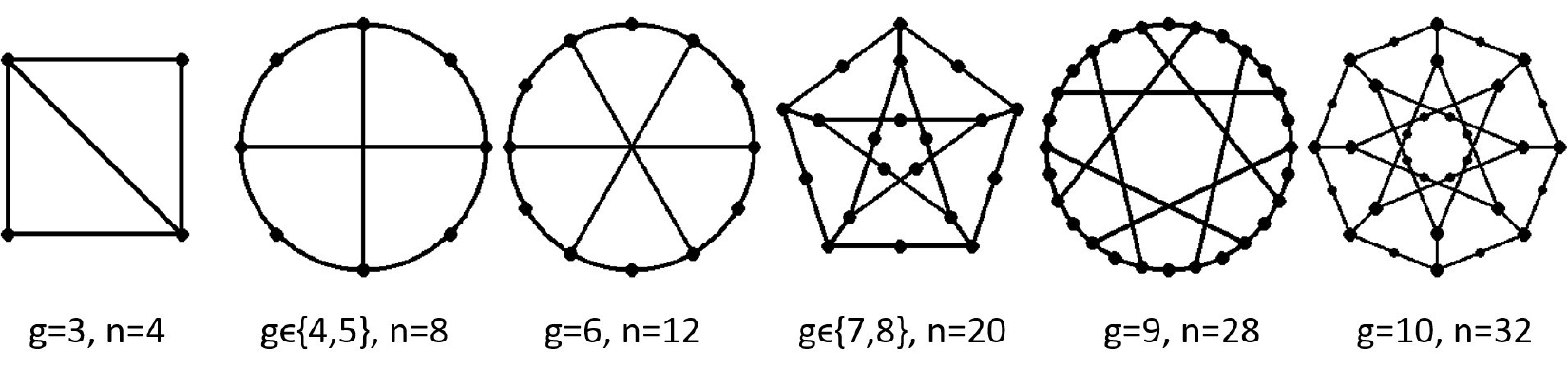}
    \caption{$(H_{23},g)$-cages for $g=3,4,\ldots,10$.}
    \label{fig:cages}
\end{figure}

\FloatBarrier

\subsection{Data Availability}\label{subsection:data}
The following data are available at \url{https://doi.org/10.5281/zenodo.23384940} and as ancillary files of the arXiv version of this paper:
\begin{enumerate}
    \item For each $g$ and each construction algorithm, and for each $g \le 15$ also the exhaustive search, the smallest lift of $H_{23}$ of girth at least $g$ that it found, as a list of edges. These graphs certify every upper bound in Table~\ref{table:results}.
    \item A short standalone program that checks that a graph is a lift of $H_{23}$ and computes its girth.
    \item The source code of our implementation of Algorithm~\ref{alg:exhaustive}.
    \item The output of the exhaustive search for every refuted case, listing the verdict for each cycle type, which proves the lower bounds of Table~\ref{table:results} that exceed the Moore bound.
\end{enumerate}

\section{Conclusion \& Open Problems}
In this work we suggest alternatives, applicable to non-regular graphs, for the well studied problem of finding $n(k,g)$, 
which is the smallest order of a $k$-regular graph of girth at least $g$.
Given a fixed base graph $H$, we define: 
(i) $n({\cal L}(H),g) = n(H,g)$ as the smallest order of a graph of girth at least $g$ covering $H$,
(ii) $n({\cal U}(H),g)$ as the smallest order of a graph of girth at least $g$ with the same universal cover as $H$,
and (iii) $n({\cal D}(H),g)$ as the smallest order of a graph of girth at least $g$ with the same proportion of each degree as $H$. 
It should be noted that the second and third definitions offer a true generalization of $n(k,g)$, by taking any $k$-regular graph as $H$.

We prove that the first two definitions are equivalent up to a multiplicative constant, 
and we show that the well known Moore lower bound and Erd\H{o}s-Sachs upper bound hold in the new setting as well,
with the exponent base being $\rho(H)$ rather than $k-1$.
For the 3rd definition we have a Moore bound result by Alon, Hoory and Linial~\cite{alon2002moore}.
However, there is no Erd\H{o}s-Sachs result and the Moore bound, with exponent base $\Lambda(H)$, 
might be far from achievable as we conjecture is the case for $H_{23}$.

The question regarding the asymptotic behavior of $n(H,g)$, or $n(k,g)$ remains largely unanswered with a large gap between the lower and upper bounds.
Constructions of better families of graphs, 
such as Lubotzky, Phillips and Sarnak~\cite{lubotzky1988ramanujan}, better graphs for small values of $g$ such as in this work for $n(H_{23},g)$ and the survey of $n(3,g)$ graphs by Exoo and Jajcay~\cite{exoo2012dynamic} point towards the conjecture that the true exponent is much closer to that of the Moore bound than to that of the best currently known constructions.

\vskip 5pt \noindent
We end this work with a list of conjectures and open problems:
\begin{enumerate}    
    \item {\bf Conjecture:} $n({\cal D}(H_{23}),g) \geq {\rho(H_{23})}^{(1+o(1))g/2}$

    Namely, the exponent base lower bound on the order of graphs of girth at least $g$ where half the vertices have degree 2 and the other half have degree 3,  can be improved from $\Lambda(H_{23})$ to $\rho(H_{23})$. An even bolder conjecture would be that graphs achieving the minimum cannot have adjacent degree 2 vertices.

    \item {\bf Conjecture:} The minimal possible $\rho$ for a graph in ${\cal D}(H_{23})$ is $\rho(H_{23})$.

    One should note that Proposition~\ref{proposition:rhoLambda} implies that $\rho(G) > \Lambda(H_{23})$ for any graph $G \in {\cal D}(H_{23})$.
    However, it falls short of proving that $\rho(G) > \Lambda(H_{23}) + \epsilon$ for some explicit $\epsilon>0$.
    
    \item {\bf Problem:} Prove that any of the greedy algorithms listed in Section~\ref{section:concrete} satisfies the asymptotic Erd\H{o}s-Sachs upper bound of Theorem~\ref{theorem:main} with high probability.

    At least for Algorithm A, one may attempt to adapt the proof by Linial and Simkin~\cite{linial2021randomized} to the non-regular case.

    \item {\bf Problem:} 
    Prove that the Moore lower bound for $H_{23}$, given by \eqref{eqn:exac_Moore_odd} and \eqref{eqn:exac_Moore_even} cannot hold as equality for large values of $g$, even when taking into account the divisibility constraint that $n(H_{23},g)$ must be divisible by four.

    \item {\bf Problem:} 
    Provide a number theoretic construction, for some non-regular graph $H$, of lifts of $H$ of girth at least $g$ and order ${\rho(H)}^{(1+o(1))\cdot 3g/4}$.
\end{enumerate}

\section*{Acknowledgements}
I would like to thank my students Michael Shevelin, Tal Shamir, Almog Elbaz and Shaked Sharon for exploring greedy algorithms to construct small girth $g$ graphs covering $H_{23}$ and improving the exhaustive search for small values of $g$ as part of an undergrad project.
Also, I would like to thank Dani Kotlar and Nati Linial for many fruitful discussions.

\bibliographystyle{acm}
\bibliography{bib}

\end{document}